\documentclass[a4paper,12pt]{amsart}

\usepackage{amsmath,amsfonts,amssymb,amsthm,amscd,latexsym,euscript}
\usepackage{mathrsfs}
\usepackage{tikz,tikz-cd}

\usepackage[all]{xy}

\usepackage{mathtools}
\usepackage{bm}

\usepackage{pdfsync}

\usepackage{todonotes}

\usepackage{appendix}

\newtheorem{theorem}{Theorem}[section]
\newtheorem{claim}[theorem]{Claim}
\newtheorem*{claim*}{Claim}
\newtheorem{corollary}[theorem]{Corollary}
\newtheorem{definition}[theorem]{Definition}

\newtheorem{lemma}[theorem]{Lemma}
\newtheorem{proposition}[theorem]{Proposition}
\newtheorem{prop}[theorem]{Proposition}
\newtheorem{remark}[theorem]{Remark}

\newtheorem{question}[theorem]{Question}
\newtheorem{theoremletter}{Theorem}

\newcommand{\Ce}{{\mathcal C}}

\newcommand{\PPP}{{\mathbb{P}}}

\renewcommand{\emptyset}{\varnothing}

\numberwithin{equation}{section}

\newcommand{\ran}[1]{{{\rm{ran}}(#1)}}
\newcommand{\crit}[1]{{{\rm{crit}}(#1)}}
\newcommand{\cof}[1]{{{\rm{cof}}(#1)}}
\newcommand{\seq}[2]{\langle{#1}~\vert~{#2}\rangle}
\newcommand{\map}[3]{{#1}:{#2}\longrightarrow{#3}}

\newcommand{\pmap}[4]{{#1}:{#2}\xrightarrow{#4}{#3}}
\newcommand{\id}{{\rm{id}}}
\newcommand{\Set}[2]{\{{#1}~ \vert~{#2}\}}
\newcommand{\POT}[1]{{\mathcal{P}}({#1})}
\newcommand{\length}[1]{{\rm{lh}}({#1})}
\newcommand{\dom}[1]{{{\rm{dom}}(#1)}}

\newcommand{\UXSR}{\mathrm{UXSR}}
\newcommand{\XSR}{\mathrm{XSR}}

\newcommand{\OD}{\mathrm{OD}}
\newcommand{\HOD}{\mathrm{HOD}}
\newcommand{\ZFC}{\mathrm{ZFC}}
\newcommand{\Ord}{{\rm{Ord}}}

\begin{document}

\title{Large cardinals beyond HOD}

\author[Aguilera]{Juan P. Aguilera}
\address{Institut f\"ur diskrete Mathematik und Geometrie, TU Wien. Wiedner Hauptstrasse 8-10, 1040 Vienna, Austria.}
\email{aguilera@logic.at}

\author[Bagaria]{Joan Bagaria}
\address{ICREA (Instituci\'o Catalana de Recerca i Estudis Avan\c{c}ats) and
\newline \indent Departament de Matem\`atiques i Inform\`atica, Universitat de Barcelona. 
Gran Via de les Corts Catalanes, 585,
08007 Barcelona, Catalonia.}
\email{joan.bagaria@icrea.cat}

\author[Goldberg]{Gabriel Goldberg}
\address{Department of Mathematics, University of California, Berkeley. 970 Evans Hall, Berkeley, CA 94720-3840 USA.}

\email{ggoldberg@berkeley.edu }

\author[L\"ucke]{Philipp L\"{u}cke}
\address{Fachbereich Mathematik, Universit\"at Hamburg, Bundesstra{\ss}e 55, Hamburg, 20146, Germany}
\email{philipp.luecke@uni-hamburg.de}

\date{\today }

%
%
%

\date{\today }

\begin{abstract}
Exacting and ultraexacting cardinals are large cardinal numbers compatible with the Zermelo-Fraenkel axioms of set theory, including the Axiom of Choice. In contrast with standard large cardinal notions, their existence implies that the set-theoretic universe $V$ is not equal to G\"odel's subuniverse of Hereditarily Ordinal Definable ($\HOD$) sets.

We prove that the existence of an ultraexacting cardinal is equi\-consistent with the well-known axiom I0; moreover, the existence of ultraexacting cardinals together with other standard large cardinals is equiconsistent with generalizations of I0 for fine-structural models of set theory extending $L(V_{\lambda+1})$. We prove tight bounds on the strength of exacting cardinals, placing them strictly between the axioms I3 and I2. The argument extends to show that I2 implies the consistency of Vop\v{e}nka's Principle together with an exacting cardinal and the HOD Hypothesis. In particular, we obtain the following result: the existence of an extendible cardinal above an exacting cardinal does not refute the HOD Hypothesis. 

We also give several new characterizations of exacting and ultraexacting cardinals; first in terms of strengthenings of the axioms I3 and I1 with the addition of Ordinal Definable predicates, and finally also in terms of principles of Structural Reflection which characterize exacting and ultraexacting cardinals as natural two-cardinal forms of strong unfoldability.
\end{abstract}
\maketitle
\setcounter{tocdepth}{1}
\tableofcontents

\section{Introduction}\label{SectIntro}
\textit{Large cardinals} are infinite cardinal numbers with structural properties making them so large that their existence cannot be proved on the basis of the standard Zermelo-Fraenkel axioms of set theory with the Axiom of Choice ($\ZFC$). Their existence nonetheless has profound consequences throughout mathematics, leading to applications in algebra (see e.g., \cite{MS:AlmostFree, EM:AFM, LAVER:LDA, Ba-Ma:GR}), algebraic topology (see e.g., \cite{BCMR}), general topology (see e.g., \cite{BdS:NMSC} for an example and \cite{Ku:handbook} for a collection of surveys), measure theory (see e.g., \cite{Solovay2, So:Mod, SW}), game theory  (see e.g., \cite{MS:PPD, Woodin:PNAS}), and category theory (see e.g., \cite{AdR:LPAC}).
In addition to their applications, large cardinals are studied for their own sake, as they are one of the main sources of insight into the nature of infinity and the structure of the universe of sets (see \cite{Kan:THI} for a general survey of large cardinals).

\textit{Exacting} and \textit{ultraexacting} cardinals are large cardinals recently introduced in \cite{ABL}. A cardinal $\lambda$ is called \textit{exacting} if for every cardinal $\zeta>\lambda$, there is an elementary substructure $X$ of $V_\zeta$ (the collection of all sets of von Neumann rank less than $\zeta$) such that $V_\lambda \cup \{\lambda\} \subseteq X$ and an elementary (i.e., truth-preserving) embedding
\[j: X \to V_\zeta\]
such that $j(\lambda) = \lambda$ and $j\upharpoonright \lambda \neq \mathrm{Id}$. We say $\lambda$ is \textit{ultraexacting} if we  additionally require that $j\upharpoonright V_\lambda \in X$. One of the most important properties of exacting and ultraexacting cardinals is that their existence implies that $V$ is not equal to HOD, the sub-universe of Hereditarily Ordinal Definable sets; in other words, they imply that there exist sets that cannot be defined in any reasonable way. This is in contrast to ``traditional'' large cardinals, all of which are consistent with the hypothesis $V = \mathrm{HOD}$. Moreover, in contrast with all large cardinals studied so far, ultraexacting cardinals become much stronger in the presence of other large cardinals, thus calling into question the established idea that large cardinals form a linear hierarchy of increasing strength. All this poses a challenge to our current  conception of large cardinals, which calls for a further study of their consistency strength, especially in combination with other large cardinals, and for obtaining further evidence of their naturalness as axioms.

The purpose of this article is twofold. First, we settle the questions concerning the  strength of exacting and ultraexacting cardinals in relation to traditional large cardinals, thereby placing them in their exact position within the current hierarchy of large cardinals, as well as the question of whether exacting cardinals suffice to replicate the ``blow-up'' behavior of ultraexacting cardinals in the presence of other large cardinals. Second, we prove new equivalences of these large cardinals with  some straightforward  enhancements of traditional large cardinals, as well as with some simple forms of reflection that generalize the reflection properties of traditional large cardinals, attesting to the naturalness of exacting and ultraexacting cardinals. 

It was proved in \cite{ABL} that ultraexacting cardinals are consistent relative to the existence of an I0 embedding, i.e., a nontrivial elementary embedding $j : L(V_{\lambda+1}) \to L(V_{\lambda+1})$ with $j\upharpoonright \lambda \neq \mathrm{Id}$. The first ordinal moved by such an embedding is a large cardinal, also known as an I0 cardinal. They were first introduced by Woodin in the 1980s to establish the consistency of the Axiom of Determinacy, and they are located at the upper end of the roster of traditional large cardinals. Our first result, proved in \S\ref{SectGood}, is the converse of this theorem, thus establishing: 
\begin{theoremletter}\label{TheoremIntroI0}
The following are equiconsistent over $\ZFC$:
    \begin{enumerate}
        \item There exists an I0 embedding.
        \item There exists an ultraexacting cardinal.
    \end{enumerate}
\end{theoremletter}

Here, recall that axioms $T_1$ and $T_2$ are said to be \textit{equiconsistent over} ZFC if the arithmetical formalization of the sentence ``ZFC + $T_1$ is consistent'' is equivalent to that of ``ZFC + $T_2$ is consistent'' (with the equivalence provable arithmetically).

Theorem \ref{TheoremIntroI0} is a corollary of a result that shows how ultraexacting cardinals
fit into a previously studied paradigm
for producing large cardinal-like hypotheses that contradict the HOD Conjecture: 
\begin{theoremletter}
    If \(\lambda\) is a cardinal, the following are equivalent:
    \begin{enumerate}
        \item \(\lambda\) is ultraexacting.
        \item For every ordinal definable \(A\subseteq V_{\lambda+1}\)
        there is a non-trivial elementary embedding from \((V_{\lambda+1},A)\) to itself.
    \end{enumerate}
\end{theoremletter}
From this perspective, one can view ultraexacting cardinals as a local version of the hypothesis that there is an elementary embedding from \(\HOD_{V_{\lambda+1}}\) to itself.
This hypothesis was considered by Woodin \cite{WOEM1},
who showed that the existence of such a \(\lambda\) above an extendible cardinal is 
consistent with ZFC relative to choiceless large cardinal assumptions \cite[p.\ 335]{WOEM1}
but contradicts the HOD Conjecture \cite[Theorem 199]{WOEM1}.
Woodin \cite[Theorem 200]{WOEM1} obtained the same conclusion using a hypothesis on \(\lambda\) that he showed was consistent with ZFC assuming \(\text{Con}(\text{ZFC} + \textnormal{I0})\); namely, the existence of an elementary embedding \(j : V_{\lambda+1}\to V_{\lambda+1}\) such that for any \(\Sigma_2\)-formula \(\varphi(x)\)
and any \(a\in V_{\lambda+1}\), \(V\vDash \varphi(a)\) if and only if \(V\vDash \varphi(j(a))\).
In Corollary \ref{coroSigma2truth}, we show that this hypothesis is equivalent to $\lambda$ being ultraexacting.

One of the most novel aspects of ultraexacting cardinals is their non-trivial interaction with other large cardinals. For instance, in contrast to Theorem \ref{TheoremIntroI0}, the existence of \textit{two} I0 cardinals is much weaker than the existence of two ultraexacting cardinals (see \cite[Theorem D]{ABL}). Our next theorem, proved in \S\ref{SectUltraexactGen} and of which Theorem \ref{TheoremIntroI0} is in fact a particular case, clarifies this phenomenon:

\begin{theoremletter}\label{TheoremIntroUltraexactGen}
Let $\varphi$ be a formula in the language of set theory. Then, the following two theories are equiconsistent, modulo $\ZFC$:
\begin{enumerate}
    \item There is an ultraexacting cardinal $\lambda$ and a countably iterable Mitchell-Steel $V_{\lambda+1}$-premouse $M$ satisfying $\varphi$.
    \item There is a countably iterable Mitchell-Steel $V_{\lambda+1}$-premouse $M$ satisfying $\varphi$ and an elementary embedding \[j:L(V_{\lambda+1},M)\to L(V_{\lambda+1},M) \] 
    with critical point below $\lambda$.
\end{enumerate}
\end{theoremletter}
According to Theorem \ref{TheoremIntroUltraexactGen}, ultraexacting cardinals stretch the strength of large cardinals above them by producing elementary embeddings which strengthen I0 by incorporating inner models into the domain of the embedding. It was proved in \cite{ABL} that the existence of an ultraexacting cardinal $\lambda$ together with $V_{\lambda+1}^\#$ implies the consistency of a proper class of I0 cardinals. Theorem \ref{TheoremIntroUltraexactGen} strengthens and generalizes this fact optimally.

We also investigate whether exacting cardinals are sufficient to replicate this ``blow-up'' phenomenon of large cardinal strength. We answer the question negatively in \S\ref{SectExacting} while simultaneously bounding the strength of exacting cardinals dramatically. Its statement involves the notion of an $\mathrm{I3_{wf(0)}}$-embedding. This is a technical strengthening of I3 much weaker than I2 (see \S\ref{SectExacting} for the definitions of I3 and I2).

\begin{theoremletter}\label{TheoremIntroExacting}
Suppose that $j: V_\lambda \to V_{\lambda}$ is an $\mathrm{I3_{wf(0)}}$-embedding with critical point $\kappa$. Then there is a normal ultrafilter $U$ on $\kappa$ such that if $G$ is a Prikry-generic $U$-sequence, then 
\[V[G]\models \text{``$\kappa$ is exacting.''}\]
\end{theoremletter}

We also show that the use of an $\mathrm{I3_{wf(0)}}$-embedding in Theorem \ref{TheoremIntroExacting} cannot be replaced by a simple I3 embedding, as the existence of exacting cardinals implies the consistency of a proper class of I3 cardinals. Thus, Theorem \ref{TheoremIntroExacting} is close to an equiconsistency and exacting cardinals are located strictly between the hypotheses I3 and I2.

Theorem \ref{TheoremIntroExacting} also yields models of exacting cardinals together with any large cardinal which is preserved by ``small'' Prikry forcing, indicating a key difference between exacting and ultraexacting cardinals.

Perhaps the most striking consequence of exacting cardinals, established in \cite{ABL} is that the consistency of an exacting cardinal above a strongly compact cardinal refutes Woodin's HOD Conjecture. In sharp contrast to this, we derive the following result from Theorem \ref{TheoremIntroExacting} (Corollary \ref{coro5.6}):

\begin{theoremletter}\label{TheoremIntroExactingwithHOD}
Suppose $\ZFC + \mathrm{I2}$ is consistent. Then, $\ZFC$ is consistent with the existence of an exacting cardinal together with Vop\v{e}nka's principle and the HOD Hypothesis.
\end{theoremletter}

In particular,
the consistency of $\mathrm{ZFC}$ together with exacting cardinal \textit{below} an extendible cardinal does not refute the HOD Conjecture (assuming the consistency of I2).

Finally, in \S\ref{SectSR} we give new characterizations of ultraexacting and exacting cardinals in terms of principles of structural reflection, improving the results in \cite{ABL}. The characterizations show that these cardinals fit nicely in the hierarchy of large cardinals when they are viewed under the general framework of structural reflection (see \cite{Ba:SR}). 

It is shown in  \cite{BL2} that the smallest $C^{(n)}$-strongly unfoldable cardinal can be characterized in terms of Structural Reflection as the smallest $\mu$ with the property that if $\mathcal{C}$ is a class of structures of the same signature, which is  $\Sigma_{n+1}$-definable from parameters in $V_\mu$, and $B \in \mathcal{C}$ has size $\mu$, then there is $A \in \mathcal{C}$ of size $<\mu$ and an elementary embedding 
\[j: A\to B,\]
provided that $\mathcal{C}$ contains some structure of size $<\mu$. In general, this characterizes $C^{(n)}$-strong unfoldability, with the exception that limits of $C^{(n-1)}$-extendible cardinals also satisfy this  property.

We end section \S\ref{SectSR} by showing (Theorem \ref{TheoremExactingUnfoldable})  that  exacting  cardinals are also characterized in terms of Structural Reflection as a two-cardinal variant of $C^{(n)}$-strong unfoldability. The characterization is obtained by adding a second cardinal constraint to the structures considered. Let $n \geq 2$. Then, $\lambda$ is exacting if and only if for some $\mu$, for every class of structures $\mathcal{C}$ of the same signature $\Sigma_n$-definable from parameters in $V_\mu \cup \{\lambda\}$ and every $B\in \mathcal{C}$ of type $\langle\mu,\lambda\rangle$, there is $A\in \mathcal{C}$ of type $\langle\nu,\lambda\rangle$ with \(\nu < \mu\) and an elementary embedding $j: A\to B$. Ultraexacting cardinals admit a similar characterization in which the embedding $j$ is required to be a square root of a fixed embedding. 
\section{Preliminaries}
Let us recall the definitions of exact and ultraexact embeddings. Recall that an embedding $j: M \to N$ is \textit{elementary} if it is truth-preserving, i.e., if for all tuples $a \in [M]^{<\mathbb{N}}$ and all formulas $\phi$, we have $M\models \phi(a)$ if and only if $N\models\phi(j(a))$. By convention, all elementary embeddings occurring in this article are assumed to be non-trivial, i.e., different from the identity.

\begin{definition}[\cite{BL, ABL}]
  Let $n>0$ be a natural number and let  $\lambda$   be a limit cardinal. 
 Given a cardinal $\lambda<\eta\in C^{(n)}$, an elementary submodel $X$ of $V_\eta$ with $V_\lambda\cup\{\lambda\}\subseteq X$, and a cardinal $\lambda<\zeta\in C^{(n+1)}$, an elementary embedding $j:X\to V_\zeta$ is an \emph{$n$-exact embedding at $\lambda$} if $j(\lambda)=\lambda$,  and $j\restriction\lambda$ is not the identity on $\lambda$. If, moreover, we require that $j\restriction V_\lambda\in X$, then we say that $j$ is an \emph{$n$-ultraexact embedding at $\lambda$}.
\end{definition}

The following lemma from \cite{ABL} shows that the notions of exact and ultraexact  embedding are independent of $n$ (for $n>0$).

\begin{lemma}\label{lemma:Equivalentexact}
Given a natural number $n>0$, the following statements are equivalent for every  limit ordinal $\lambda$ and every set $x$: 
\begin{enumerate}
    \item There is an $n$-exact ($n$-ultraexact) embedding $j:X\to V_\zeta$ at $\lambda$ with $x\in X$ and $j(x)=x$. 
    \item There are elements $\eta$ and $\zeta$ of $C^{(2)}$ greater than $\lambda$, an elementary submodel $X$ of $V_\eta$ with $V_\lambda\cup\{\lambda,x\}\subseteq X$, and an elementary embedding $j:X\to V_\zeta$ with $j(\lambda)=\lambda$, $j(x)=x$, $j\restriction\lambda \neq {\rm Id}_\lambda$ (and $j\restriction V_\lambda \in X$).
    \item For every $\zeta>\lambda$ with $x\in V_\zeta$, there is an elementary submodel $X$ of $V_\zeta$ with  $V_\lambda\cup\{\lambda,x\}\subseteq X$, and an elementary embedding $j:X\to V_\zeta$ with $j(\lambda)=\lambda$, $j(x)=x$, $j\restriction \lambda \neq {\rm Id}_\lambda$  (and $j\restriction V_\lambda \in X$). 
\end{enumerate}
\end{lemma}

 Thus, an $n$-exact ($n$-ultraexact)  embedding at $\lambda$ exists (for $n>0$) if and only if for some $\zeta \in C^{(2)}$ greater than $\lambda$ (equivalently, for the least such $\zeta$), there is an elementary embedding $j:X\to V_\zeta$, where $X$ is an elementary substructure of $V_\zeta$ that contains $V_\lambda \cup \{ \lambda\}$, such that $j(\lambda)=\lambda$, $j\restriction \lambda \ne {\rm Id}_\lambda$ (and $j\restriction V_\lambda \in X$).  

In view of Lemma \ref{lemma:Equivalentexact}, we shall say that an embedding $j:X\to V_\zeta$ is an \emph{exact} (\emph{ultraexact})  embedding at $\lambda$ if it is an $n$-exact ($n$-ultraexact) embedding at $\lambda$, for some $n>0$. Now we define:

\begin{definition}\label{defn:exacting}
A cardinal $\lambda$ is \emph{exacting} (\emph{ultraexacting}) if there exists an exact (ultraexact) embedding at $\lambda$.   
\end{definition}

Thus, $\lambda$ is exacting (ultraexacting) if and only if there is an elementary embedding $j:X\to V_\zeta$, where $\zeta$ is some cardinal in $C^{(2)}$ above $\lambda$ (equivalently, the least such) and $X$ is an elementary substructure of $V_\zeta$ that contains $V_\lambda \cup \{ \lambda\}$, such that  $j(\lambda)=\lambda$, $j\restriction \lambda \ne {\rm Id}_\lambda$ (and $j\restriction V_\lambda \in X$). Equivalently, $\lambda$ is exacting (ultraexacting) if such an embedding
exists for all $\zeta>\lambda$.

The following proposition shows that our definition of exacting  and ultraexacting cardinals are equivalent to the definitions given in \cite[2.4, 3.3]{ABL}.

\begin{prop}
 A cardinal $\lambda$ is \emph{exacting} (\emph{ultraexacting}) iff for every $\zeta$ in $C^{(2)}$ above $\lambda$ (equivalently, the least such) and every $\alpha <\lambda$  there is an elementary embedding $j:X\to V_\zeta$, where  $X$ an elementary substructure  $V_\zeta$ that contains $V_\lambda \cup \{ \lambda\}$,  such that  $j(\lambda)=\lambda$, $j\restriction \alpha = {\rm Id}_\alpha$, $j\restriction \lambda \ne {\rm Id}_\lambda$ (and $j\restriction V_\lambda \in X$).     
\end{prop}

\begin{proof}
 Suppose, aiming for a contradiction, that for some $\zeta \in C^{(2)}$ above $\lambda$, for some $\alpha <\lambda$ the required embedding does not exist. Let $\alpha$ be the least witness.   Let $\zeta'$ be the first cardinal in $C^{(2)}$ above $\zeta$. Let $j: X\to V_{\zeta'}$ witness that $\lambda$ is exacting (ultraexacting). As $\alpha$ and $\zeta$ are definable in $V_{\zeta'}$, both $\alpha$ and $\zeta$ belong to $X$ and they are fixed by $j$. Then $X\cap V_\zeta$ is an elementary substructure of $V_\zeta$ and $j\restriction X\cap V_\zeta :X\cap V_\zeta \to V_\zeta$ witnesses that $\lambda$ is exacting (ultraexacting), with $j\restriction \alpha= {\rm Id}_\alpha$, contrary to our assumption on $\alpha$. 
\end{proof}

Recall that an \emph{${\rm I}0$ embedding} (at a cardinal $\lambda$) is an 
elementary embedding $$j:L(V_{\lambda +1})\to L(V_{\lambda +1})$$ with critical point less than $\lambda$. Axiom ${\rm I}0$ asserts that there exists an ${\rm I}0$ embedding (definable, possibly from parameters).

We shall show (Theorem \ref{TheoremEquiconsistencyI0}) that the existence of an ultraexacting cardinal is equiconsistent with the existence of an ${\rm I}0$ embedding. But notice that since being an ultraexacting cardinal is $\Sigma_3$-expressible, the first ultraexacting cardinal, if it exists, is below the first extendible cardinal.

\section{Large cardinals and ordinal definable sets} \label{SectOD}
Recall that an ${\rm I}1$ embedding is an elementary embedding from $V_{\lambda +1}$ to itself, with $\lambda$ a limit ordinal. If $j:L(V_{\lambda +1})\to L(V_{\lambda +1})$ is an ${\rm I}0$ embedding, then clearly $j\restriction V_{\lambda +1}$ is an ${\rm I}1$ embedding. Further, if $A$ is subset of $V_{\lambda +1}$ that is definable in $L(V_{\lambda +1})$ with parameters in $V_{\crit{j}}\cup \{\lambda\}$, then $j\restriction V_{\lambda +1}:(V_{\lambda +1}, A)\to (V_{\lambda +1}, A)$ is an elementary embedding.

We shall prove next a characterization of exacting and ultraexacting cardinals in terms of ${\rm I}1$ embeddings expanded with ordinal definable predicates.
We prove this characterization  from $\mathrm{ZF} + \mathrm{DC}_\lambda$, rather than $\ZFC$. Recall that $\mathrm{DC}_\lambda$ (or $\lambda{-}\mathsf{DC}$) is the assertion that every ${<}\lambda$-closed tree with no terminal points has a branch of length $\lambda$.
Moreover, recall that $\OD$ is the class of all \emph{Ordinal Definable} sets. It is well known that $\OD$ is a $\Sigma_2$ class, and that there exists a $\Sigma_2$-definable well-ordering, $<_\OD$, of  $\OD$ (see, e.g., \cite{Jech}).

\begin{theorem}[$\mathrm{ZF} + \mathrm{DC}_\lambda$]
    \label{lemma:ultra_exact_char}
   A cardinal $\lambda$ is ultraexacting if and only if for every ordinal definable subset $A$ of $V_{\lambda+1}$, there exists an elementary embedding
        \(j: (V_{\lambda+1},A) \to (V_{\lambda +1} , A)\).
        \end{theorem}
    \begin{proof}
    First, assume $\lambda$ is ultraexacting and, towards a contradiction, that   
        \(A\) is the \(<_\OD\)-least subset of \(V_{\lambda+1}\) such that there is no elementary embedding
        from \((V_{\lambda+1},A)\) to itself.
        Fix \(\zeta\in C^{(2)}\) and  \(X\preceq V_{\zeta}\) such that \(V_\lambda \cup \{\lambda\}\subseteq X\), and
        \(j : X \to V_\zeta\) is an elementary embedding with \(\crit{j} < \lambda\), $j(\lambda)=\lambda$,  and \(j\restriction V_\lambda \in X\).
         Note that \(A\in V_\zeta\)
        and \(A\) is definable in \(V_\zeta\) from \(\lambda\),
        and therefore \(A\in X\) and \(j(A) = A\).
        Let \(k = j\restriction V_{\lambda+1}\). Then \(k\in X\), since \(j\restriction V_\lambda \in X\),
        and \(k\) is definable from \(j\restriction V_\lambda\) by $k(x)= \bigcup \{ j(x\cap V_\alpha):\alpha <\lambda\}$. 

        The main point is that, in \(X\), \(k\) is an elementary embedding from \((V_{\lambda+1},A)\) to itself.
        The latter claim follows from the fact that for all \(x\in V_{\lambda+1}\cap X\), 
        $$X\vDash ``(V_{\lambda+1},A)\vDash \varphi(x)"$$
        if and only if $$V_\zeta\vDash ``(V_{\lambda+1},A)\vDash \varphi(k(x))"$$
        by the elementarity of \(k\), and this holds if and only if 
        $$X\vDash ``(V_{\lambda+1},A)\vDash \varphi(k(x))"$$ since \(k(x)\in X\) and
        \(X\preceq V_\zeta\). This contradicts the choice of \(A\) as
        the \(\OD\)-least subset of $V_{\lambda+1}$ such that there is no elementary embedding
        from \((V_{\lambda+1},A)\) to itself.
        
        \medskip

For the converse, let $\zeta$ be the least element of $C^{(2)}$ above $\lambda$, and let $A$ be the set of all well-founded extensional relations $E\subseteq V_{\lambda}\times V_{\lambda}$ such that $\langle V_\lambda , E\rangle$ is isomorphic to some elementary substructure $X\preceq V_\zeta$ with $V_{\lambda}\cup \{\lambda\}\subseteq X$. Then $A$ is $\Delta_2$-definable with $\zeta$ and $\lambda$ as parameters. Note that $A$ is non-empty. To prove this, first observe that $\mathrm{DC}_\lambda$ implies that $V_\alpha$ can be wellordered for each $\alpha<\lambda$ (by induction on $\alpha$); using this and $\mathsf{DC}$, we see that $|V_\lambda| = \lambda$. Using this fact, $\mathrm{DC}_\lambda$ allows us to carry out the proof of the L\"owenheim-Skolem theorem to construct elementary substructures of $V_\zeta$ containing $V_\lambda$. Thus, indeed $A$ is nonempty.
        
Now, let $j : (V_{\lambda+1}, A)\to (V_{\lambda+1},A)$ be an elementary embedding and fix $E\in A$ such that $j\restriction V_\lambda$ belongs to the transitive collapse $M_E$ of $\langle V_\lambda ,E\rangle$. Let $F = j(E)$ and let $M_F$ be the transitive collapse of $\langle V_\lambda , F\rangle$. Then $j \restriction M_E : M_E\to M_F$ is elementary. Moreover there are $X_E,X_F\preceq V_\zeta$, both  including $V_{\lambda}\cup \{\lambda\}$, and isomorphisms $\pi_E: X_E \cong M_E$  and $\pi_F: X_F\cong M_F$. Now letting ${\rm Id}_{X_F}:X_F\to V_\zeta$ be the identity map, and letting 
        $$i:= {\rm Id}_{X_F}\circ \pi^{-1}_F\circ j\restriction M_E \circ \pi_E$$
        we have that $i:X_E\to V_\zeta$ is an elementary embedding that agrees with $j$ on $V_{\lambda}$, and moreover $i\restriction V_\lambda\in X_E$. Thus, $X_E$ and $i$ witness that $\lambda$ is ultraexacting.
\end{proof}

The characterization of ultraexacting cardinals given by Theorem \ref{lemma:ultra_exact_char} makes no reference to elementary substructures of $V_\zeta$, so it motivates the following re-definition of ultraexacting cardinals, which is the definition we shall use in the context where $\mathrm{DC}_\lambda$ fails:

\begin{definition}[$\mathrm{ZF}$] \label{DefUltraexactingZF} A cardinal $\lambda$ is \emph{ultraexacting} if and only if for every ordinal definable $A\subset V_{\lambda+1}$ there is an elementary embedding $j:(V_{\lambda+1},A) \to (V_{\lambda+1},A)$ with critical point ${<}\lambda$.
\end{definition}

\begin{remark}\label{remark1}
The proof of Theorem \ref{lemma:ultra_exact_char} above gives some additional information. Namely, given $\lambda$, if $\zeta$ is the least element of $C^{(2)}$ greater than $\lambda$, then the following are equivalent:
\begin{enumerate}
    \item $\lambda$ is ultraexacting.
    \item There is an elementary embedding from $(V_{\lambda +1}, A)$ to itself, where $A$ is the subset of $V_{\lambda +1}$ used in the proof of the theorem above and is $\Delta_2$-definable from the parameters $\lambda$ and $\zeta$. Namely, $A$ is the set of all well-founded extensional relations $E\subseteq V_{\lambda}\times V_{\lambda}$  isomorphic to some elementary substructure $X\preceq V_\zeta$ with $V_{\lambda}\cup \{\lambda\}\subseteq X$.
\end{enumerate}
\end{remark}

We also have the following equivalence in the case of exacting cardinals:

\begin{theorem}[$\mathrm{ZF} + \mathrm{DC}_\lambda$]\label{lemma:exact_char}
A cardinal $\lambda$ is exacting
if and only if or every non\-empty ordinal definable   subset $A$ of $V_{\lambda+1}$, there exist $x,y\in A$ and an elementary embedding $j: (V_{\lambda },x)\to (V_{\lambda },y)$.
\end{theorem}

\begin{proof}
Suppose $\lambda$ is exacting and, towards a contradiction, let $A$ be the $\OD$-least nonempty subset of $V_{\lambda +1}$ for which there are no $x,y\in A$ with an elementary embedding $j: (V_{\lambda},x)\to (V_{\lambda},y)$. Let $\zeta$ be the least ordinal in $C^{(2)}$ greater than the least ordinal parameters appearing in a definition of $A$. Then fix $X\preceq V_\zeta$ and $j : X\to V_\zeta$ as in the definition of exacting cardinal, and note that $A\in X$ and $j(A)= A$. Therefore for any $x\in A\cap X$, setting $y = j(x)$ and taking the restriction $j\restriction V_\lambda: (V_\lambda , x)\to (V_\lambda , y)$, we obtain a contradiction.

The converse is proved similarly as in the previous theorem, using the same $A$.  As $A$ is nonempty, let $E, F\in A$ be such that there is an elementary embedding $j : (V_{\lambda}, E)\to (V_{\lambda},F)$.  There are $X_E,X_F\preceq V_\zeta$, both  including $V_{\lambda}\cup \{\lambda\}$, and isomorphisms $\pi_E: X_E \cong M_E$  and $\pi_F: X_F\cong M_F$, with $M_E$ and $M_F$ transitive. Then letting $i$ be as before, 
        we have that  $X_E$ and $i$ witness that $\lambda$ is exacting.
\end{proof}

As in the case of ultraexacting cardinals, the theorem above  motivates the following re-definition of exacting cardinals, which may be used in the context where $\mathrm{DC}_\lambda$ fails: 
\begin{definition}[$\mathrm{ZF}$]\label{def:exacting-zf} A cardinal $\lambda$ is \emph{exacting} if and only if for every nonempty ordinal definable $A\subset V_{\lambda+1}$ there are $x,y\in A$ and an elementary embedding $j:(V_{\lambda},x) \to (V_{\lambda},y)$ with critical point $<\lambda$.
\end{definition}

Similar considerations, as in Remark \ref{remark1}, also apply in this case. Namely,

\begin{remark}\label{remark2}
Given any cardinal $\lambda$, if $\zeta$ is the least element of  $C^{(2)}$ greater than $\lambda$, then the  following are equivalent:
\begin{enumerate}
    \item $\lambda$ is exacting.

    \item There is an elementary embedding $j:(V_{\lambda}, x)\to (V_{\lambda}, y)$, where $x,y$ belong to the subset $A\subseteq V_{\lambda +1}$ which is  used in the proof of the theorem and is $\Delta_2$-definable from the parameters $\lambda$ and $\zeta$. 
\end{enumerate}
\end{remark}

\subsection{Some corollaries}

We shall next obtain several corollaries of Theorem \ref{lemma:ultra_exact_char} above. The first one shows that the existence of an ultraexacting cardinal $\lambda$ is equivalent to the existence of an elementary embedding  $j : V_{\lambda +1}\to V_{\lambda+1}$ that preserves the $\Sigma_2$-truth predicate of $V$, an axiom first considered by Woodin in \cite[Theorem 200]{WOEM1}.

\begin{corollary}
\label{coroSigma2truth}
A cardinal $\lambda$ is ultraexacting if and only if there exists an elementary embedding $j : V_{\lambda +1}\to V_{\lambda+1}$ such that for any \(\Sigma_2\)-formula \(\varphi(x)\) and any \(a\in V_{\lambda+1}\),
\(V\vDash \varphi(a)\) if and only if \(V\vDash \varphi(j(a))\).
\end{corollary}

\begin{proof}
If $\lambda$ is ultraexacting, then since the restriction \(A\) of the \(\Sigma_2\)-satisfaction predicate of \(V\) to \(V_{\lambda+1}\) is ordinal definable,  Theorem  \ref{lemma:ultra_exact_char} yields the desired elementary embedding.

Conversely, suppose there is an elementary $j : V_{\lambda +1}\to V_{\lambda+1}$ such that for any \(\Sigma_2\)-formula \(\varphi(x)\) and any \(a\in V_{\lambda+1}\),
\(V\vDash \varphi(a)\) if and only if \(V\vDash \varphi(j(a))\). Suppose towards a contradiction that there is an ordinal definable set \(A\) for which there is no elementary embedding \(i : (V_{\lambda+1},A)\to (V_{\lambda+1},A)\). Let \(A\) be the OD-least such set. Then the satisfaction predicate \(S\) of \((V_{\lambda+1},A)\) is \(\Sigma_2\)-definable over \(V\) from the parameter \(\lambda\). Fix a \(\Sigma_2\)-formula \(\varphi(x)\) such that \(u\in S\) if and only if \(V\vDash \varphi(u,\lambda)\).
Then by our hypothesis, \(u\in S\) if and only if \(V\vDash \varphi(u,\lambda)\) if and only if \(V\vDash \varphi(i(u),\lambda)\) if and only if \(i(u)\in S\). Since \(j\) maps the satisfaction predicate of \((V_{\lambda+1},A)\) into itself, \(j\) is an elementary embedding from \((V_{\lambda+1},A)\) to itself, contrary to our hypothesis that no such embedding exists.
\end{proof}

Recall from \cite[Definition 132]{WOEM1} that a cardinal $\delta$ is \emph{$\HOD$-super\-compact} if for all $\eta >\delta$ there exists an elementary embedding $j:V\to M$, $M$ transitive, with critical point $\delta$, $j(\delta)>\eta$, $^{V_\eta}M\subseteq M$, and $j(\HOD \cap V_\delta)\cap V_\eta = \HOD \cap V_{\eta}$.
Woodin's Theorem \cite[200]{WOEM1} then shows that, assuming the $\HOD$ Conjecture,  the existence of a nontrivial elementary embedding from $V_{\lambda+1}$ to itself that preserves the $\Sigma_2$-truth predicate of $V$ implies there is no $\HOD$-supercompact cardinal  below $\lambda$. He also notes that such an embedding exists in the forcing extension of $L(V_{\lambda+1})$ that well-orders $V_{\lambda+1}$ in order-type $\lambda^+$. 
Then the contrapositive of Woodin's theorem, together with Corollary \ref{coroSigma2truth}, yields that if $\lambda$ is an ultraexacting cardinal, and there is a $\HOD$-supercompact cardinal below $\lambda$, then the $\HOD$ Conjecture fails.    However, a stronger result follows from 
\cite[section 6.1]{ABL} together with Goldberg's \cite[Section 2.2]{GG} which shows  that Woodin's $\HOD$ Dichotomy follows from the existence of a strongly compact cardinal, yielding that if there exists a strongly compact cardinal below an exacting cardinal, then the $\HOD$ Conjecture fails.

\medskip

Theorem \ref{lemma:ultra_exact_char} also yields a simpler proof of the following result from \cite[Theorem 3.22]{ABL}:
\begin{corollary}
    If \(\lambda\) is an ultraexacting cardinal and \(V_{\lambda+1}^\#\) exists\footnote{Recall that the existence of $V_{\lambda +1}^\#$  is equivalent to the existence of an elementary embedding $j:L(V_{\lambda +1})\to L(V_{\lambda +1})$ with critical point above $\lambda$.}, then \({\rm I}0\) holds at \(\lambda\).
    \begin{proof}
        First, recall that, by Solovay,  \(V_{\lambda+1}^\#\) exists if and only if there exists a \emph{remarkable character} for $V_{\lambda +1}$ (see, e.g.,  \cite{C:sharps}), and so \(V_{\lambda+1}^\#\) can be coded as an ordinal definable subset, $A$, of $V_{\lambda +1}$. Now by Theorem \ref{lemma:ultra_exact_char}, ultraexactness yields an elementary embedding \(j\) from \((V_{\lambda+1},A)\)
        to itself. The embedding \(j\) can then be extended to an elementary embedding
        \(i\) from \(L(V_{\lambda+1})\) to itself
        by setting \(i(t(x,\overline \xi)) = t(j(x),\overline \xi)\)
        whenever \(t\) is a canonical weak Skolem term, 
        \(x\) is an element of \(V_{\lambda+1}\), and \(\overline \xi\) is a finite increasing
        sequence of Silver indiscernibles for \(L(V_{\lambda+1})\). It is then easily checked  that \(i\) is well-defined
        and elementary, and so \(i\) witnesses that \({\rm I}0\) holds at \(\lambda\).
    \end{proof}
\end{corollary}

\begin{proposition}\label{ultraexacting+forcing}
    If $\lambda$ is an ultraexacting cardinal, then for every cardinal $\gamma >\lambda$ 
    there is a $\gamma$-directed closed forcing notion which forces  that $\lambda$ remains ultraexacting and 
    $V_\lambda\subseteq \HOD$.
    \begin{proof}
    Let \(\zeta\in C^{(2)}\) and  \(X\preceq V_{\zeta}\) be such that \(V_\lambda \cup \{\lambda\}\subseteq X\), and
        \(j : X \to V_\zeta\) is an elementary embedding with \(\crit{j} < \lambda\), $j(\lambda)=\lambda$,  and \(j\restriction V_\lambda \in X\).
        Let $\lhd$ be a wellordering of $V_\lambda$ of order-type $\lambda$ such that $j(\lhd)=\lhd$. 
This can be obtained as follows:  Let $\langle \lambda_n :n<\omega\rangle$ be the critical sequence of $j$. Pick a wellordering $\lhd_0$ of $V_{\lambda_0}$, and let $\lhd_1=j(\lhd_0)\setminus \lhd_0$. Given $\lhd_m$ for some $0<m<\omega$, set  $\lhd_{m+1}=j(\lhd_m)$. Finally, let $\lhd=\bigcup_{m<\omega}\lhd_m$.

Note that, since $j\restriction V_\lambda \in X$, $\lhd \in X$. So, as $j(\lhd)=\lhd$, arguing  like in the proof of Theorem \ref{lemma:ultra_exact_char}, using $\lhd$ as a parameter, we have that for all sets $A\subseteq V_{\lambda+1}$ that are ordinal definable from $\lhd$, there is an elementary embedding from $ (V_{\lambda+1},A)$ to itself.

        Given any $\gamma >\lambda$, we may  code  $\lhd$ into the power-set function on the regular cardinals above $\gamma$ by a $\gamma$-directed closed homogeneous forcing that is ordinal definable from $\lhd$. Then in the generic extension, $V[G]$,
        $\lhd$ is ordinal definable,  hence $V_\lambda\subseteq \HOD.$ 
        
       By the homogeneity of the forcing, if $A\subseteq V_{\lambda+1}$ is ordinal definable  in $V[G]$, then $A$ is in \(V\) and \(A\) is ordinal definable from $\lhd$ in $V$. Thus, there is, in $V$ and therefore also in $V[G]$, an elementary embedding from $(V_{\lambda +1}, A)$ to itself. Hence by Theorem \ref{lemma:ultra_exact_char},  $\lambda$ is ultraexacting in $V[G]$.
    \end{proof}
\end{proposition}

It was shown in \cite[2.10]{ABL} that if $\lambda$ is exacting, then $\lambda$ is regular in HOD. Thus, $V_{\lambda+1} \nsubseteq \mathrm{HOD}$, so Proposition \ref{ultraexacting+forcing} is best possible.

 \begin{corollary}
      If $\lambda$ is ultraexacting, then in some forcing extension $\lambda$ remains ultraexacting and there is no exacting cardinal below $\lambda$.      
        \end{corollary}

        \begin{proof}
        By Proposition \ref{ultraexacting+forcing}, let $V[G]$ be a forcing extension in which $\lambda$ remains ultraexacting and $V_\lambda =V_\lambda^{V[G]} \subseteq \HOD$. Then in $V[G]$ no cardinal $\mu <\lambda$ can be exacting, as it would imply that $\mu$ is a regular cardinal in $\HOD_{V_\mu}$ (by \cite[2.10]{ABL}), and therefore also in $V[G]$.    
        \end{proof}

\section{The strength of ultraexacting cardinals}\label{SectGood}
In this section, we prove  equiconsistencies between ultraexacting cardinals and strengthenings of I0, depending on the kinds of large cardinals which exist in inner models extending $V_{\lambda+1}$. In particular, we establish the equiconsistency between the existence of an ultraexacting cardinal and I0. 

We work with models of the form $L(V_{\lambda+1}, E)$, where $E \subseteq V_{\lambda+1}$, assuming $\ZFC$ holds externally in $V$. All what follows is true also in the particular case $E = \varnothing$.
These models have a fine structure similar to that of $L(V_{\lambda+1})$, just like $L(x)$ has a fine structure similar to that of $L$ when $x \in\mathbb{R}$, and just like how $L(\mathbb{R},A)$ has a fine structure similar to that of $L(\mathbb{R})$ when $A\subseteq \mathbb{R}$ and $\mathrm{DC}_\mathbb{R}$ holds.

For an ordinal $\alpha$, let $A_\alpha^E$ be the theory of $L_\alpha (V_{\lambda +1},E)$, with parameters in $V_{\lambda +1}$. Thus, $A_\alpha^E$ may be identified with the set of all pairs $(\varphi, a)$ such that $a\in V_{\lambda +1}$, $\varphi$ is a formula of the language of set theory with an added predicate $\dot E$ and with only one free variable, and $L_\alpha (V_{\lambda +1},E )\models \varphi(a)$. Notice that $A^E_\alpha\in V_{\lambda +2}$.

Let us call an ordinal $\alpha$ an \emph{$E$-good ordinal} (or just \emph{good}, if $E$ is clear from context; see Laver \cite{La:Refl}) if every element of  $L_\alpha (V_{\lambda +1},E)$ is definable in $L_\alpha (V_{\lambda +1},E)$ from parameters in $V_{\lambda +1}$ in the language $\mathcal{L}_{\in, \dot E}$. It is easily seen that every good ordinal is strictly less than $\Theta^{L(V_{\lambda +1},E)}$, where $\Theta^{L(V_{\lambda +1},E)}$ is the supremum of the set of ordinals $\gamma$ such that there exists a surjection $f:V_{\lambda +1} \to \gamma$ with $f\in L(V_{\lambda +1})$. Moreover, the argument of \cite[Lemma 1]{La:Refl} shows that the good ordinals are unbounded in $\Theta^{L(V_{\lambda +1},E)}$.

\begin{lemma}
\label{lemma1}\label{lemma_good_ord}
Suppose that $\lambda$ is ultraexacting and $E\subseteq V_{\lambda+1}$ is ordinal definable.
Then for every $E$-good ordinal $\alpha$ there exists an elementary  embedding $i:L_\alpha (V_{\lambda +1},E)\to L_\alpha (V_{\lambda +1},E)$, with $\lambda$ being the supremum of its critical sequence, and moreover $i \in L(V_{\lambda+1},E)$ and $i(E) = E$.   
\end{lemma}
\begin{proof}
Since \(A^E_\alpha\) is ordinal definable, Theorem \ref{lemma:ultra_exact_char} implies that there is a nontrivial elementary embedding \(j : (V_{\lambda+1},A^E_\alpha)\to (V_{\lambda+1},A^E_\alpha)\). We extend \(j\) to an elementary embedding \(i : L_\alpha (V_{\lambda +1},E) \to L_\alpha (V_{\lambda +1},E)\)
using the fact that every element of \(L_\alpha (V_{\lambda +1},E)\) is definable in \(L_\alpha (V_{\lambda +1},E)\) from parameters in $V_{\lambda +1}$ in the language $\mathcal{L}_{\in, \dot E}$.
More precisely, if \(a\) is definable in \(L_\alpha (V_{\lambda +1},E)\) from parameters $x\in V_{\lambda +1}$ in the language $\mathcal{L}_{\in, \dot E}$, let
\(i(a)\) be the element of \(L_\alpha(V_{\lambda +1},E)\) defined by the same formula from \(j(x).\) This definition immediately yields $i(E) = E$.

The fact that \(i :  L_\alpha (V_{\lambda +1},E)\to L_\alpha (V_{\lambda +1},E)\) is well-defined and elementary is immediate from the elementarity of \(j\) on \((V_{\lambda+1},A^E_\alpha)\). Moreover, one can verify that \(i\restriction V_{\lambda+1}= j\) since each \(x\in V_{\lambda+1}\) is trivially definable from itself. Finally, \(i\in L(V_{\lambda+1},E)\) since \(i\) is definable over \(L_{\alpha+1}(V_{\lambda+1},E)\) from parameters in \(\{j,E\}\subseteq L_1(V_{\lambda+1},E)\). 
\end{proof}

\subsection{Internal ${\rm I0}$}\label{SectInternalI0}
Let $E\subseteq V_{\lambda+1}$.
We say that \textit{internal \({\rm I}0\) relative to $E$ holds at \(\lambda\)} 
if for all \(\alpha < \Theta^{L(V_{\lambda+1},E)}\),
there is an elementary embedding from \(L_\alpha(V_{\lambda+1},E)\) to itself fixing \(E\), with $\lambda$ the supremum of its critical sequence. We speak of \textit{internal I0} rather than ``internal I0 relative to $\varnothing$.''

Internal I0 was first isolated by Woodin in \cite{WoBrink}. The point is that while \({\rm I}0\) at $\lambda$  
cannot hold in \(L(V_{\lambda+1})\), internal \({\rm I}0\) is \textit{absolute} between
\(L(V_{\lambda+1})\) and \(V\). 
Indeed, internal \({\rm I}0\) relative to $E\subseteq V_{\lambda+1}$ is absolute between
\(L(V_{\lambda+1},E)\) and \(V\). The reason is that there exist arbitrarily large good ordinals $\alpha < \Theta^{L(V_{\lambda+1},E)}$. For such $\alpha$, the existence of elementary embeddings from $L_\alpha(V_{\lambda+1},E)$ to itself fixing \(E\) is equivalent to the existence of elementary embeddings from $(V_{\lambda+1}, A^E_\alpha)$ to itself. If these exist, then they can be recovered from their restriction to $V_\lambda$ as in the proof of Lemma \ref{lemma_good_ord} and thus belong to $L(V_{\lambda+1}, E)$.
Moreover, the existence of elementary embeddings from $L_\alpha(V_{\lambda+1},E)$ to itself fixing \(E\) for arbitrarily large \(\alpha < \Theta^{L(V_{\lambda+1},E)}\) easily implies the existence of such embeddings for all \(\alpha <  \Theta^{L(V_{\lambda+1},E)}\) by a minimization argument. 

Woodin \cite{WoBrink} showed
that the theory $\ZFC + ``{\rm I}0 \mbox{ at }\lambda$" is conservative over the theory $\ZFC + ``\text{Internal \({\rm I}0\) at $\lambda$"}$ 
for first-order statements about \(L(V_{\lambda+1})\).  

Lemma \ref{lemma_good_ord}, together with Laver's argument \cite[Lemma 1]{La:Refl} that the good ordinals are unbounded in $\Theta^{L(V_{\lambda +1},E)}$, yields the following:

\begin{corollary}\label{CorollaryUltraexactingImpliesInternalI0}
Suppose that \(\lambda\) is ultraexacting and $E\subseteq V_{\lambda+1}$ is ordinal definable. Then internal \({\rm I}0\) relative to $E$ holds at \(\lambda\), both in $V$ and in $L(V_{\lambda+1}, E)$.
\end{corollary}

The following lemma yields the converse implication in  mild forcing extensions of $L(V_{\lambda +1},E)$.

\begin{lemma}\label{CorollaryInternalI0ImpliesUltraexacting}
Let $E\subseteq V_{\lambda+1}$.
Suppose that internal \({\rm I}0\) relative to $E$ holds at \(\lambda\).
Then, \(\lambda\) is ultraexacting in any generic extension of \(L(V_{\lambda+1},E)\) by an ordinal definable
homogeneous forcing notion that does not change \(V_{\lambda+1}\).
\begin{proof}
We first show that \(\lambda\) is ultraexacting in \(L(V_{\lambda+1},E)\). It suffices to show that for any \(A\subseteq V_{\lambda+1}\) that is \(\text{OD}\) in \(L(V_{\lambda+1},E)\), there is a nontrivial elementary embedding from \((V_{\lambda+1},A)\) to itself. 
By condensation, any such \(A\) is definable from \(E\) in \(L_\gamma(V_{\lambda+1}, E)\) for some \(\gamma < \Theta = \Theta_{V_{\lambda+1}}^{L(V_{\lambda+1},E)}\), and 
by Internal I0 there is an elementary embedding from \(L_{\gamma}(V_{\lambda+1},E)\) to itself fixing \(E\), with critical point below \(\lambda\);
this restricts to a nontrivial elementary embedding from \((V_{\lambda+1},A)\) to itself (which belongs to \(L(V_{\lambda+1},E)\) since it is induced by its restriction to \(V_\lambda\)). This shows that $\lambda$ is ultraexacting in $L(V_{\lambda+1}, E)$.
        
Now suppose that \(\mathbb P\in \OD^{L(V_{\lambda+1},E)}\)
is a homogeneous forcing notion that does not change \(V_{\lambda+1}\), and \(G\subseteq \mathbb P\) is \(L(V_{\lambda+1},E)\)-generic. Then every set \(A\subseteq V_{\lambda+1}\) that is OD in \(L(V_{\lambda+1},E)[G]\) is already OD in \(L(V_{\lambda+1},E)\), so the fact that \(\lambda\) is ultraexacting in \(L(V_{\lambda+1},E)\) immediately implies that it is ultraexacting in \(L(V_{\lambda+1},E)[G]\).
    \end{proof}
\end{lemma}

Since Lemma \ref{lemma_good_ord} holds in $\mathrm{ZF} + \mathrm{DC}_\lambda$, Lemmata \ref{lemma_good_ord} and \ref{CorollaryInternalI0ImpliesUltraexacting} yield the following:

\begin{corollary}
 In $L(V_{\lambda +1})$, $\lambda$ is an ultraexacting cardinal if and only if internal \({\rm I}0\) holds.   
\end{corollary}

Lemma \ref{CorollaryInternalI0ImpliesUltraexacting} sharpens {\cite[Theorem 3.29]{ABL}}: if there is an ${\rm I}0$ embedding at $\lambda$, then after forcing with ${\rm Add}(\lambda^+,1)$ over $V$, in $L(V_{\lambda +1})[G]$  the cardinal $\lambda$ is ultraexacting and $\ZFC$ holds. Moreover, if $\lambda$ is ultraexacting, then, by Corollary \ref{CorollaryUltraexactingImpliesInternalI0}  internal ${\rm I}0$ holds at $\lambda$, and \cite[Theorem 2.5]{WoBrink} shows that in some generic extension of an inner model of $L(V_{\lambda +1})$, $\ZFC$ holds and there is an ${\rm I}0$ embedding at a cardinal greater than $\lambda$. Thus, we have the following:

\begin{theorem}\label{TheoremEquiconsistencyI0}
The following two theories are equiconsistent:
\begin{enumerate}
    \item     \textnormal{ZFC +} There exists an ultraexacting  cardinal.
    \item  \textnormal{ZFC +} \({\rm I}0\) holds.  
    \end{enumerate}
\end{theorem}

Let us note that the existence of an ultraexacting cardinal does not suffice to prove  ${\rm I}0$ outright, 
since, as mentioned above, if  $\lambda$ is the least such that ${\rm I}0$ holds at $\lambda$, then after forcing with ${\rm Add}(\lambda^+,1)$ over $V$, $\lambda$ is ultraexacting in $L(V_{\lambda +1})[G]$ by \cite[Theorem 3.29]{ABL}, yet \({\rm I}0\) fails there. To see this, first note that by the minimality of \(\lambda\), \({\rm I}0\) does not hold for any \(\gamma < \lambda\), since this is absolute between \(V\) and \(L(V_{\lambda+1})[G]\); and obviously \({\rm I}0\) does not hold for any \(\gamma> \lambda\); finally \({\rm I}0\) fails at $\lambda$, since $L(V_{\lambda +1})[G]$ is a set forcing extension of \(\text{HOD}^{L(V_{\lambda+1})}\) by Vopenka's theorem \cite[Theorem 15.46]{Jech} and therefore cannot contain an elementary embedding from  \(\text{HOD}^{L(V_{\lambda+1})}\) to itself by a result of Hamkins--Kirmayer--Perlmutter \cite[Corollary 9]{HamkinsKirmayerPerlmutter}.

\subsection{$V_{\lambda+1}$-premice}
The statement of Theorem \ref{TheoremEquiconsistencyGeneralized} below involves the notion of a Mitchell-Steel $V_{\lambda+1}$-premouse. We do not assume familiarity with inner model theory and indeed essentially only use it to derive the conclusion of Lemma \ref{LemmaMisOD} below. Nonetheless, we summarize the relevant definitions.

Let $X$ be a set. A Mitchell-Steel $X$-premouse is a structure of the form $J_\alpha(X)[E]$ (i.e., a level of the Jensen hierarchy constructed relative to $X$ and $E$) where $E$ is a predicate for a \textit{fine extender sequence} over $V_{\lambda+1}$, in the sense of Steel \cite[Definition 2.6]{St08}.

Given a Mitchell-Steel premouse $M$, one can define \textit{iteration trees} $\mathcal{T}$ over $M$. These are trees of iterated ultrapowers of very particular kinds. A
premouse is \textit{countably iterable} if all its countable elementary substructures are $(\omega_1+1)$-iterable. The reader who is not familiar with this notion might simply elect to take Lemma \ref{LemmaMisOD} and the first couple of lines in the proof of Theorem \ref{TheoremEquiconsistencyGeneralized} on faith; we refer the reader to Steel \cite{St10} for background. For definiteness, let us specify that ``premouse'' and other inner model theoretic notions should be taken as defined in \cite{St10}, although what follows is not too sensitive to the precise definitions used. The relativized notion of a $V_{\lambda+1}$-premouse is defined as in \cite{St08}.
We also refer to ``soundness'' (i.e., ``$\omega$-soundness'') of premice in the sense of \cite{St10}. This is a condition on the relation between the structure and its partial Skolem hulls. We will not make use of the definition directly, but only indirectly via comparison arguments (see e.g., \cite[Corollary 10]{St10}).

\begin{lemma}\label{LemmaMisOD}
Suppose that $M$ is a countably iterable, sound Mitchell-Steel $V_{\lambda+1}$-premouse satisfying $\varphi$ but none of whose proper initial segments satisfy $\varphi$. Then, $M$ is definable from $\lambda$ and there is a surjection from $V_{\lambda+1}$ to $M$ definable over $M$.
\end{lemma}
\begin{proof}
This is a routine argument, but we sketch it for the reader's convenience. First, that there is a surjection from $V_{\lambda+1}$ to $M$ definable over $M$ follows from a Skolem hull argument as in the case of $L$. To see that $M$ is ordinal definable, we show that it is unique. Suppose towards a contradiction that $M$ and $M'$ are two different countably iterable, sound Mitchell-Steel premice satisfying $\varphi$ and none of whose proper initial segments satisfy $\varphi$. Standard arguments now lead to a contradiction: let $H$ be the transitive collapse of a countable elementary substructure of some large $V_\theta$ containing $M$ and $M'$ and let $\pi:H \to V_\theta$ be the collapse embedding. Let $A = \pi^{-1}(V_{\lambda+1})$, $N_0 = \pi^{-1}(M)$, and $N_1 = \pi^{-1}(M')$. Then, $N_0$ and $N_1$ are $A$-premice, $N_0 \neq N_1$, and $N_0$ and $N_1$ are $(\omega_1+1)$-iterable. Moreover, $N_0$ and $N_1$ are sound and project to $A$.
By the comparison theorem for Mitchell-Steel premice (see Steel \cite[\S 3.2]{St10}), it follows that one of $N_0$ or $N_1$ is a proper initial segment of the other, contradicting the fact that neither has a proper initial segment satisfying $\varphi$. This proves that $M' = M$.
\end{proof}

\subsection{Ultraexacting cardinals in the presence of  other large cardinals}\label{SectUltraexactGen}
The goal of this section is to extend the equiconsistency proof of Theorem \ref{TheoremEquiconsistencyI0} to describe the strength of ultraexacting cardinals in the presence of other large cardinals. 
We shall need the following version of Woodin's  \cite[Theorem 2.3]{WoBrink} for models of the form $L(V_{\lambda+1},M)$. 

\begin{lemma}\label{LemmaBrink}
Suppose that $M\subseteq V_{\lambda +1}$ and internal I0 relative to $M$ holds in $L(V_{\lambda+1}, M)$ at $\lambda$.
Then there is an elementary embedding \[j: (V_{\lambda+1}, M) \to (V_{\lambda+1},M)\] 
such that, letting $(N_\omega, \bar M, j_\omega)$ be the $\omega$th iterate of $j$, $j_\omega$ extends to an elementary embedding 
\[k: L(N_\omega, \bar M) \to L(N_\omega, \bar M)\]
such that $L(N_\omega, \bar M)[k] \subseteq L(V_{\lambda+1}, M)$ and $\mathcal{P}(N_\omega) \cap L(N_\omega, \bar M)[k] = \mathcal{P}(N_\omega) \cap L(N_\omega, \bar M)$.
\end{lemma}
\begin{proof}
This is proved by the same argument as Theorem 2.3 in Woodin \cite{WoBrink} (which is stated in the particular case where $M = \varnothing$). 
\end{proof}

We now state our equiconsistency result:

\begin{theorem}\label{TheoremEquiconsistencyGeneralized}
Let $\varphi$ be a formula in the language of set theory. Then, the following two theories are equiconsistent, modulo $\ZFC$:
\begin{enumerate}
    \item \label{TheoremEquiconsistencyGeneralized1} There is an ultraexacting cardinal $\lambda$ and a countably iterable Mitchell-Steel $V_{\lambda+1}$-premouse $M$ satisfying $\varphi$.
    \item\label{TheoremEquiconsistencyGeneralized2} There is a countably iterable Mitchell-Steel $V_{\lambda+1}$-premouse $M$ satisfying $\varphi$ and an elementary embedding \[j:L(V_{\lambda+1},M)\to L(V_{\lambda+1},M) \] 
    with critical point below $\lambda$.
\end{enumerate}
\end{theorem}
\proof
First, assume that \eqref{TheoremEquiconsistencyGeneralized2} holds and let $\Theta = \Theta_{V_{\lambda+1}}^{L(V_{\lambda+1},M)}$. By replacing $M$ with an initial segment if necessary, we may assume that $M$ has no proper initial segment satisfying $\varphi$. Replacing $M$ with its core if necessary, we may assume that $M$ is sound.
By Lemma \ref{LemmaMisOD}, we may thus replace $M$ by a code in $V_{\lambda+2}$ definable from $\lambda$; abusing notation, we denote this code by $M$ too.  By hypothesis, there is an elementary embedding 
\[j:L(V_{\lambda+1},M)\to L(V_{\lambda+1},M) \] 
with critical point below $\lambda$. Since $M$ is definable from $\lambda$ only we must have $j(M) = M$, so internal I0 relative to $M$ holds in $V$ and thus also in $L(V_{\lambda+1}, M)$, as this principle is absolute (see the comment at the beginning of \S \ref{SectInternalI0}). 

It follows from Lemma \ref{CorollaryInternalI0ImpliesUltraexacting} that if $G\subseteq \text{Add}(\lambda^+,1)$ is $V$-generic, then $\lambda$ is ultraexacting in $L(V_{\lambda+1},M)[G]$. Since  $\text{Add}(\lambda^+,1)$ does not change $V_{\lambda+1}$ and in particular it does not add any new countable elementary substructures of $M$ or any new iteration trees of length $\leq \omega_1$, $M$ remains countably iterable in $L(V_{\lambda+1},M)[G]$ and thus $L(V_{\lambda+1},M)[G]$ is a model of $\ZFC$ satisfying \eqref{TheoremEquiconsistencyGeneralized1}.\\

We now suppose that \eqref{TheoremEquiconsistencyGeneralized1} holds. As before, we may assume that $M \in V_{\lambda+2}$ and that $M$ is definable from the parameter $\lambda$. By Lemma \ref{lemma_good_ord}, internal I0 relative to $M$ holds in $L(V_{\lambda+1}, M)$.

Let $k$ be as in Lemma \ref{LemmaBrink} and let $\lambda_\omega = j_\omega(\lambda)$. Thus, the following hold in $L(N_\omega, \bar M)[k]$:
\begin{enumerate}
    \item ZF + $j_{\omega}(\lambda)$-DC,
    \item $k: L(N_\omega, \bar M) \to L(N_\omega, \bar M)$ is an elementary embedding,
    \item $\bar M$ is a Mitchell-Steel $(V_{\lambda_\omega+1})^{L(N_\omega, \bar M)[k]}$-premouse satisfying $\varphi$ and $\bar M$ is countably iterable.
\end{enumerate}
The first item follows from the elementarity of $j_\omega$; the second, from the choice of $k$; the third, from the elementarity of $j_\omega$ together with the fact that $N_\omega$ contains all sets of rank below the critical point of $j$ and in particular is correct about the countable iterability of $\bar M$. Let $G\subseteq \text{Add}(\lambda_\omega^+,1)$ be $L(N_\omega, \bar M)[k]$-generic and consider the model 
\[W = L(N_\omega, \bar M)[k][G].\]
Then, $W \models\ZFC$. Working in $W$, the forcing does not change $V_{\lambda_\omega+1}$, and thus $\bar M$ remains a countably iterable $(V_{\lambda_\omega+1})^{W}$-premouse satisfying $\varphi$ and $W$ satisfies that 
\[k: L(V_{\lambda_\omega+1},\bar M) \to L(V_{\lambda_\omega+1},\bar M)\]
is an elementary embedding with critical point $<\lambda_\omega$, so \eqref{TheoremEquiconsistencyGeneralized2} holds in $W$. This completes the proof of Theorem \ref{TheoremEquiconsistencyGeneralized}. 
\endproof

We mention as corollaries some particular cases of Theorem \ref{TheoremEquiconsistencyGeneralized} which might be of interest. First, the case where $\varphi$ is $0 = 0$ is Theorem \ref{TheoremEquiconsistencyI0}.
The second one improves \cite[Theorem D]{ABL}:
\begin{corollary}
The following are equiconsistent:
\begin{enumerate}
    \item There is an ultraexacting cardinal $\lambda$ such that $V_{\lambda+1}^\#$ exists; and 
    \item \textup{I0}$^\#$, i.e., there is an elementary embedding \[j:L(V_{\lambda+1},V_{\lambda+1}^\#)\to L(V_{\lambda+1},V_{\lambda+1}^\#)\] 
with critical point below $\lambda$.
\end{enumerate}
\end{corollary}

The third corollary we mention is an equiconsistency result which gauges the strength of Woodin cardinals above an ultraexacting cardinal: 
\begin{corollary}\label{CorollaryWoodinEquiconsistency}
The following schemata are equiconsistent as $n$ ranges over elements of $\mathbb{N}$:
\begin{enumerate}
\item $\lambda$ is ultraexacting and there are $n$ Woodin cardinals greater than $\lambda$; and
\item there is a transitive  model $M$  and an elementary embedding $j: M \to M$ such that $V_{\lambda+1} \in M$, $\crit j<\lambda$, and there are $n$ Woodin cardinals above $\lambda$ in $M$.
\end{enumerate}
\end{corollary}
\begin{proof}
By Theorem \ref{TheoremEquiconsistencyGeneralized}, the consistency of an ultraexacting cardinal $\lambda$ below $n+1$ Woodin cardinals implies the consistency of an elementary embedding 
\[ j:L(V_{\lambda+1},M_{n}^\#(V_{\lambda+1}))\to L(V_{\lambda+1},M_{n}^\#(V_{\lambda+1}))\] 
with critical point below $\lambda$. By restricting $j$, we obtain an elementary embedding
\[k: M_{n}(V_{\lambda+1}) \to M_{n}(V_{\lambda+1})\]
as desired. 

Conversely, suppose that there is a transitive model $M$ and an elementary embedding $j: M \to M$ such that $V_{\lambda+1} \in M$, $\crit j<\lambda$, and there are $n+1$ Woodin cardinals above $\lambda$ in $M$. Let $\bar M$ be the result of carrying out the Mitchell-Steel \cite{MiSt94} construction of $L[E]$ within $M$ and relativized to $V_{\lambda+1}$. By the main theorem of Steel \cite{St93}, $\bar M$ is a countably iterable $V_{\lambda+1}$-premouse with $n+1$ Woodin cardinals above $\lambda$. Moreover, a comparison argument shows that $M_n^\#(V_{\lambda+1})$ is an initial segment of $\bar M$. By Lemma \ref{LemmaMisOD}, $M_n^\#(V_{\lambda+1})$ is definable from $\lambda$, so $j(M_n^\#(V_{\lambda+1})) = M_n^\#(V_{\lambda+1})$; thus $j$ restricts to an elementary embedding 
\[ k:L(V_{\lambda+1},M_{n}^\#(V_{\lambda+1}))\to L(V_{\lambda+1},M_{n}^\#(V_{\lambda+1})),\] 
from which the desired equiconsistency follows now by an application of Theorem \ref{TheoremEquiconsistencyGeneralized}.
\end{proof}

We remark that ``Woodin cardinals'' in the statement of Lemma \ref{CorollaryWoodinEquiconsistency} could be replaced by measurable cardinals, or by any large cardinal for which an inner model theory has been developed.


\section{The strength of exacting cardinals}\label{SectExacting}

We now investigate the consistency strength of exacting cardinals. We shall see that it lies strictly between the principles I3 and I2. Recall that I3 asserts the existence of a nontrivial elementary embedding $j:V_\lambda \to V_\lambda$ with $\lambda$ a limit ordinal. I2 asserts that for some $\lambda$ there exists an elementary embedding $j:V \to M$ with $M$ transitive, $V_\lambda\subseteq M$, $j\restriction \lambda \ne {\rm Id}$, and $j(\lambda)=\lambda$.

Let  $\map{j}{V_\lambda}{V_\lambda}$ be an I3-embedding with critical sequence $\seq{\kappa_m}{m<\omega}$. We then define 
\begin{align*}
j^+: V_{\lambda+1} &\to V_{\lambda+1}\\
A &\mapsto 
\bigcup_{\alpha<\lambda}j(A\cap V_\alpha).    
\end{align*}
It is well-known that this map is $\Sigma_0$-elementary (see \cite[Lemma 3.2]{MR3902806}). 
Then there exists a unique commuting system $$\seq{\map{j_{m,n}}{V_\lambda}{V_\lambda}}{m\leq n<\omega}$$ of elementary embeddings with    $j_{0,1}=j$, $j_{n,n}=\id_{V_\lambda}$  and $$j_{n+1,n+2} ~ = ~ j^+(j_{n,n+1}) ~ = ~ j_{n,n+1}^+(j_{n,n+1})$$ for all $n<\omega$ (see \cite[Lemmas 3.4 \& 3.5]{MR3902806}). 
 Moreover, if $n<\omega$, then $j_{n,n+1}$ is an I3-embedding with critical sequence $\seq{\kappa_{m+n}}{m<\omega}$. 
 In particular, we have $j_{n,n+k}(\kappa_{m+n})=\kappa_{m+n+k}$ for all $k,m,n<\omega$.  
 We now let $$\langle M^j_\omega, ~ \seq{\map{j_{n,\omega}}{V_\lambda}{M^j_\omega}}{n<\omega}\rangle$$ denote the direct limit of the above system and we let $W^j_\omega$ denote the well-founded part   of this model. In the following, we always identify $W^j_\omega$ with its transitive collapse. Easy computations now show that $V_\lambda\cup\{\lambda\}\subseteq W^j_\omega$ and $j_{0,\omega}(\kappa_0)=\lambda$.

 \begin{proposition}\label{proposition:FixedSets}
  Let $\map{j}{V_\lambda}{V_\lambda}$ be an $\mathrm{I3}$-embedding  with critical sequence $\seq{\kappa_m}{m<\omega}$. 
  \begin{enumerate}
   \item\label{item:FixedSet1} If $A\in V_{\lambda+1}$ is such that  $j^+(A)=A$,  then $A=j_{m,\omega}(A\cap V_{\kappa_m})\in W^j_\omega$, for every $m<\omega$. 
   
   \item\label{item:FixedSet2} If $A\in V_{\kappa_0+1}$, then $j_{0,\omega}(A)\in V_{\lambda+1}\cap W^j_\omega$ and $j^+(j_{0,\omega}(A))=j_{0,\omega}(A)$. 
   \end{enumerate}
 \end{proposition}
 
 \begin{proof}
  \eqref{item:FixedSet1} Fix $A\in V_{\lambda+1}$ with $j^+(A)=A$. Since {\cite[Lemma 3.4]{MR3902806}} ensures that $$j_{n+1,n+2}^+(A) ~ = ~ (j^+(j_{n,n+1}))^+(A) ~ = ~ j^+(j_{n,n+1}^+(A))$$ holds for all $n<\omega$, an easy induction shows that $j_{n,n+1}^+(A)=A$ holds for all $n<\omega$. 
  Another application of {\cite[Lemmas 3.4]{MR3902806}} then shows that $j_{m,n}^+(A)=A$ holds for all $m\leq n<\omega$. 
  This directly implies that $j_{m,n}(A\cap V_{\kappa_m})=A\cap V_{\kappa_n}$ holds for all $m\leq n<\omega$. 
  Since we have $j_{n,\omega}(\kappa_n)=\lambda\in W^j_\omega$ and $j_{n,\omega}\restriction V_{\kappa_n}=\id_{V_{\kappa_n}}$ for all $n<\omega$, it now follows that $j_{m,\omega}(A\cap V_{\kappa_m})=A\in W^j_\omega$ for all $m<\omega$.

  \eqref{item:FixedSet2} Fix  $A\in V_{\kappa_0+1}$. 
   Then $j(j_{0,0}(A))=j(A)=j_{0,1}(A)$ and, if $j(j_{0,n}(A))=j_{0,n+1}(A)$ holds for some $n<\omega$, then $$j(j_{0,n+1}(A))  =  j^+(j_{n,n+1})(j(j_{0,n}(A)))  =  j_{n+1,n+2}(j_{0,n+1}(A))  =  j_{0,n+2}(A).$$ This shows that $j(j_{0,n}(A))=j_{0,n+1}(A)$ holds for all $n<\omega$. 
  Additionally, if $n<\omega$, then  the fact that $j_{n,\omega}\restriction V_{\kappa_n}=\id_{\kappa_n}$ ensures that $j_{0,\omega}(A)\cap V_{\kappa_n}=j_{0,n}(A)$.  In combination, this shows that $$j(j_{0,\omega}(A)\cap V_{\kappa_n}) ~ = ~ j_{0,n+1}(A) ~ = ~ j_{0,\omega}(A)\cap V_{\kappa_{n+1}}$$ holds for all $n<\omega$ and we can conclude that $j^+(j_{0,\omega}(A))=j_{0,\omega}(A)$.  
 \end{proof}

\subsection{$\mathrm{I3}_{\rm{wf}(n)}$-embeddings}
 Following  \cite[Section 3]{MR3902806}, we say that an I3-embedding $\map{j}{V_\lambda}{V_\lambda}$   is an $\mathrm{I3}_1$-embedding if it is $\omega$-iterable, i.e., if $M^j_\omega=W^j_\omega$. Note that if $\map{j}{V_\lambda}{V_\lambda}$ is an $\mathrm{I3}_1$-embedding, then there exists a limit ordinal $\lambda^\prime<\lambda$ and an I3-embedding $\map{i}{V_{\lambda^\prime}}{V_{\lambda^\prime}}$ (see \cite[Theorem 4.1]{MR3902806}). 
In the following we will obtain exacting cardinals from the following assumption.

\begin{definition}
  Given $n<\omega$,  an $\mathrm{I3}$-embedding $\map{j}{V_\lambda}{V_\lambda}$ with critical sequence $\seq{\kappa_m}{m<\omega}$ is an {$\mathrm{I3}_{\rm{wf}(n)}$-embedding} if $j_{n,\omega}[\kappa_n^+]\subseteq W^j_\omega$. 
\end{definition}

According to the following proposition, the existence of an {$\mathrm{I3}_{\rm{wf}(n)}$-embedding} is strictly weaker than an {$\mathrm{I3}_{1}$-embedding}. Below, given a transitive set $M$, we say that a map $i: M \to M$ is a partial elementary embedding if $i$ is an embedding with domain $D\subset M$ and $i:D \to D'$ is elementary, where $D' = \bigcup i[D]$ is the codomain.
 
 \begin{proposition}
  \begin{enumerate}
   \item\label{item:HierarchyWF1}  If $m<n<\omega$ and $\map{j}{V_\lambda}{V_\lambda}$ is an $\mathrm{I3}_{\rm{wf}(n)}$-embedding, then  there exists a limit ordinal $\lambda^\prime<\lambda$ and an $\mathrm{I3}_{\rm{wf}(m)}$-embedding $\map{i}{V_{\lambda^\prime}}{V_{\lambda^\prime}}$. 

   \item\label{item:HierarchyWF2} If $\map{j}{V_\lambda}{V_\lambda}$ is an $\mathrm{I3}_1$-embedding, then for every $n<\omega$ there exists a limit ordinal $\lambda^\prime<\lambda$ and an $\mathrm{I3}_{\rm{wf}(n)}$-embedding $\map{i}{V_{\lambda^\prime}}{V_{\lambda^\prime}}$. 
  \end{enumerate}
 \end{proposition}
 
 \begin{proof}
   \eqref{item:HierarchyWF1} Fix natural numbers $m<n$ and  let $\map{j}{V_\lambda}{V_\lambda}$ be an $\mathrm{I3}_{\rm{wf}(n)}$-embedding with critical sequence $\seq{\kappa_\ell}{\ell<\omega}$. Assume, towards a contradiction, that there is no $\mathrm{I3}_{\rm{wf}(m)}$-embedding $\map{j}{V_{\lambda^\prime}}{V_{\lambda^\prime}}$ with $\lambda^\prime<\lambda$. 
   Set $\gamma=\sup(j_{m,n}[\kappa_m^+])$. 
   Since $m<n$, we have $\gamma<\kappa_n^+$. 
We define a tree $T$ whose nodes are pairs $\langle i, r\rangle$ where $\pmap{i}{V_{\kappa_n}}{V_{\kappa_n}}{part}$ is a partial elementary embedding and $r$ is an order-preserving mapping on ordinals such that there are:
   \begin{itemize}
       \item a  natural number $k> m$, 

       \item a strictly increasing sequence $\seq{\mu_\ell}{\ell\leq k+1}$ of cardinals less than $\kappa_n$, 

       \item a sequence $\seq{\map{i_\ell}{V_{\mu_k}}{V_{\mu_{k+1}}}}{\ell<k}$ of elementary embeddings, 
       
       \item a sequence $\seq{D_\ell\subseteq \mu_\ell^+}{m\leq\ell<k}$, and 

   \end{itemize}
   such that the following statements hold: 
   \begin{enumerate}
       \item\label{item:TreeDef1} $\dom{i}=V_{\mu_k +1}$, $i\restriction\mu_0=\id_{\mu_0}$ and $i(\mu_\ell)=\mu_{\ell+1}$ for all $\ell\leq k$. 

       \item\label{item:TreeDef2} $i_0=i\restriction V_{\mu_k}$ and $i_{\ell+1}=i(i_\ell\restriction V_{\mu_{k-1}})$ for all $\ell\leq k$. 

       \item\label{item:TreeDef3} $D_m=\mu_m^+$ and $D_{\ell+1}=\Set{\xi<\mu_{\ell+1}}{\exists \zeta\in D_\ell ~ \xi\leq i_\ell(\zeta)}$ for all $m\leq\ell<k-1$. 

       \item\label{item:TreeDef4} $r : D_{k-1} \to \gamma$ is an order-preserving function such that $r\upharpoonright \mu_m^+ = \id_{\mu_m^+}$ and $r(\zeta) = r(i_\ell(\zeta))$ whenever $\zeta \in D_\ell$ and $\ell < k-1$.
   \end{enumerate}
   
The ordering on $T$ is the natural one: $\langle i, r\rangle < \langle\hat\imath,\hat r\rangle$ whenever $\langle i, r\rangle, \langle\hat\imath,\hat r\rangle \in T$, $\hat\imath$ extends $i$, and $\hat r$ extends $r$. Thus, $T$ is a tree of height at most $\omega$. 

\begin{claim*}
Suppose $i\in T$. Then, the number $k$ and the sequences of cardinals $\mu_\ell$, embeddings $i_\ell$, and the sets $D_\ell$ are uniquely determined by $i$.
\end{claim*}
\begin{proof}[Proof of the Claim]
Observe first that the corresponding sequence of cardinals $\mu_l$ is uniquely determined as the (finite) critical sequence of $i$. The embeddings $i_\ell$ are also obtained from $i$ and the critical sequence by definition of $T$. This uniquely determines each $D_\ell$ as well.
\end{proof}
It follows from the claim that if
$\hat\imath$ is an extension of $i$ in $T$, then $\hat\imath$ must have a strictly longer critical sequence. Moreover, the definition of $T$ imposes  agreement on the embeddings $i_\ell$ and sets $D_\ell$.

\begin{claim*}
The tree $T$ is well-founded. 
\end{claim*}

\begin{proof}[Proof of the Claim]
Assume, towards a contradiction, that there is a branch $B$ of order-type $\omega$ through $T$. 
The union of the first and second components of $B$ yields an embedding $i$ and an order-preserving mapping $r$.
By the comment immediately before the claim, there is a cardinal $\lambda^\prime<\kappa_n$ with the property that $\map{i}{V_{\lambda^\prime}}{V_{\lambda^\prime}}$ is an I3-embedding. Let $\seq{\mu_\ell}{\ell<\omega}$ denote the critical sequence of $i$. 
Let us denote by $D$ the set of all elements of $M^i_\omega$, say $i_{\ell,\omega}(\xi)$ with  $m\leq\ell<\omega$ and $\xi<\lambda^\prime$, for which there is $\zeta<\mu_m^+$ with $\xi\leq i_{m,\ell}(\zeta)$. Observe that this is precisely the union of the sets $D_\ell$ determined by $B$. 

Similarly, $\map{r}{D}{\gamma}$ has the property that $$(r\circ i_{\ell,\omega})(\xi_0) ~ < (r\circ i_{\ell,\omega})(\xi_1)$$ holds for all $m\leq\ell<\omega$, $\zeta<\mu_\ell^+$ and $\xi_0<\xi_1<i_{m,\ell}(\zeta)$. 
    But, this directly implies that $i_{m,\omega}[\mu_m^+]\subseteq W^i_\omega$ and hence $i$ is an $\mathrm{I3}_{\rm{wf}(m)}$-embedding, contradicting our initial assumption. 
   \end{proof}

Since $T$ has cardinality at most $\kappa_n$, the above claim shows that there is an ordinal $\rho<\kappa_n^+$ and an order-reversing (ranking) function, say $\map{\pi}{T}{\rho}$. 
Now, set $\gamma_*=j_{n,\omega}(\gamma)$, $\rho_*=j_{n,\omega}(\rho)$, $\pi_*=j_{n,\omega}(\pi)$ and $T_*=j_{n,\omega}(T)$. Our setup then ensures that $\gamma_*$, $\rho_*$, $\pi_*$ and $T_*$ are all elements of $W^j_\omega$ and this directly implies that $T_*$ is a well-founded tree.

\begin{claim*}
If $k>m$ is a natural number, then $(j\upharpoonright V_{\kappa_k+1}, r_k)$ is an element of $T_*$ for some $r_k$. 
\end{claim*}
\begin{proof}[Proof of the Claim]
We need to define $r_k$ and show that $(j\upharpoonright V_{\kappa_k+1}, r_k) \in T_*$. This will be witnessed by:
\begin{itemize}
\item the natural number $k$, 
\item the sequence $\seq{\kappa_\ell}{\ell\leq k+1}$, 
\item the sequence $\seq{(j_{\ell,\ell+1}\restriction V_{\kappa_k}) : V_{\kappa_k} \to V_{\kappa_{k+1}}}{\ell< k}$, and 
\item the sequence $\seq{D_\ell}{m\leq\ell<k}$, where $D_m = \kappa_m^+$ and, for $m < \ell < k$, we have $D_\ell = \{\xi: \exists \zeta < \kappa_m^+\, \xi\leq j_{m,\ell}(\zeta)\} = \bigcup j_{m,\ell}[\kappa_m^+]$. 
\end{itemize}
The function $r_k : D_{k-1} \to \gamma_*$ from the statement of the claim is defined inductively, assuming $r_{k-1}$ has been defined the same way and that the pair $(j\upharpoonright V_{\kappa_{k-1}+1}, r_{k-1})$ has been shown to belong to $T_*$. First, we set $r_k(\xi) = \xi$ for $\xi<\kappa_m^+$. If $\xi \in D_\ell$ for $\ell<k-1$, then we set $r_k(\xi) = r_{k-1}(\xi)$; otherwise if $\xi \in D_{k-1} \setminus D_{k-2}$ (and $m < k-1$), we define
$r_k(\xi) = j_{k-1,\omega}(\xi) < \gamma_*$.
Note that $j_{\ell,\omega}[D_\ell]\subseteq\gamma_*$holds for all $m\leq\ell<\omega$. Since $\gamma_* \in W^j_\omega$, it follows that $r_k$ is order-preserving.

    %
     %
%

Let us check the
$(j\upharpoonright V_{\kappa_k+1}, r_k)$ satisfies all the clauses defining $T_*$ in $M^j_\omega$. 
For \eqref{item:TreeDef1}, we directly see that $j\upharpoonright V_{\kappa_k+1}$ has domain $V_{\kappa_k+1}$, critical point $\kappa_0$, and critical sequence $\langle \kappa_\ell: \ell \leq k\rangle$. Clause \eqref{item:TreeDef2} follows from the fact that $j_{\ell+1,\ell+2} \upharpoonright V_{\kappa_k} = j(j_{\ell,\ell+1}\upharpoonright V_{\kappa_{k-1}})$ for each $\ell < k$.
Clause
\eqref{item:TreeDef3} holds by definition. Finally, for \eqref{item:TreeDef4}, let $\xi \in D_\ell$ and suppose that $\ell$ is least such. If $\ell < k-2$, then the fact that $r_k(\xi) = r_k(j_{\ell,\ell+1}(\xi))$ follows from the inductive construction of $r_k$ and the inductive assumption that $(j\upharpoonright V_{\kappa_{k-1}+1}, r_{k-1}) \in T_*$.
Otherwise, if $\ell = k-2$, then $j_{k-2,k-1}(\xi) \in D_{k-1}\setminus D_{k-2}$, and so according to the definition of $r_k$  we have 
\begin{align*}
r_k(j_{k-2,k-1}(\xi)) = j_{k-1,\omega}(j_{k-2,k-1}(\xi)) = j_{k-2,\omega}(\xi) = r_{k-1}(\xi) = r_k(\xi).
\end{align*}
We had already checked that $r_k$ is order-preserving. This shows that $(j\upharpoonright V_{\kappa_k+1}, r_k) \in T_*$, as desired.
     %
\end{proof}
   
The above claim now directly yields a contradiction, because $T_*$ is well-founded.

Part \eqref{item:HierarchyWF2} of the proposition follows directly from part \eqref{item:HierarchyWF1},as every $\mathrm{I3}_1$-embed\-ding is an $\mathrm{I3}_{\rm{wf}(n)}$-embedding for all natural numbers $n$. 
  \end{proof}

\subsection{Exacting cardinals and Prikry forcing}
 Given a normal ultrafilter $U$ on a cardinal $\kappa$, we let $\PPP_U$ denote the corresponding Prikry forcing (see \cite[Section 18]{Kan:THI}).

 \begin{theorem}\label{theorem:PrikryExacting}
  If $\map{j}{V_\lambda}{V_\lambda}$  is an $\mathrm{I3}_{\rm{wf}(0)}$-embedding with critical sequence $\vec{\kappa}=\seq{\kappa_m}{m<\omega}$, $$U_0 ~ = ~ \Set{A\subseteq\kappa_0}{\kappa_0\in j(A)}$$ is the normal ultrafilter on $\kappa_0$ induced by $j$ and $G$ is $\PPP_{U_0}$-generic over $V$, then $\kappa_0$ is an exacting cardinal in $V[G]$. 
 \end{theorem}
 
 \begin{proof}
   Assume, towards a contradiction, that $\kappa_0$ is not exacting in $V[G]$. Then the weak homogeneity of $\PPP_{U_0}$ ensures that every condition in this partial order forces that $\lambda$ is not exacting. 
    Fix $\alpha>\lambda$ with the property that $V_\alpha$ is sufficiently elementary in $V$ and pick an elementary substructure $X$ of $V_\alpha$ of cardinality $\kappa_0$ with $V_{\kappa_0}\cup\{U_0\}\subseteq X$.  
   Let $\map{\pi}{X}{N_0}$ denote the corresponding transitive collapse. Then $V_{\kappa_0}\cup\{\kappa_0\}\subseteq N_0$. 
   Next, set $\bar{U}_0=\pi(U_0)=N_0\cap U_0\in N_0$. 
   Fix a bijection $\map{b_0}{\kappa_0}{N_0}$ with $b_0(0)=\kappa_0$, $b_0(1)=\bar{U}_0$ and $b_0(\omega\cdot\beta)=\beta$ for all $\beta<\kappa_0$. 
   Finally, let  $E_0$ be the unique well-founded and extensional relation on $\kappa_0$ with the property that $N_0$ is the transitive collapse of $\langle\kappa_0,E_0\rangle$.

  Now, set $E=j_{0,\omega}(E_0)$, $N=j_{0,\omega}(N_0)$, $U=j_{0,\omega}(U_0)$, $\bar{U}=j_{0,\omega}(\bar{U}_0)$ and $b=j_{0,\omega}(b_0)$. 
  Then $E\in V_{\lambda+1}\cap W^j_\omega$ is a binary relation on $\lambda$. Moreover, since $N_0\cap\Ord\in\kappa_0^+$, it follows that 
$$(N\cap\Ord)^{M^j_\omega} ~ = ~ j_{0,\omega}(N_0\cap\Ord) ~ \in ~ j_{0,\omega}[\kappa_0^+] ~ \subseteq ~ W^j_\omega.$$ 
  Since $M^j_\omega$ is a model of $\ZFC$, we now know that $N\in W^j_\omega$ is a transitive set with $V_\lambda\cup\{\bar{U},\lambda\}\subseteq N$ and $\map{b}{\langle\lambda,E\rangle}{\langle N,\in\rangle}$ is an isomorphism with $b(0)=\lambda$, $b(1)=\bar{U}$ and $b(\omega\cdot\beta)=\beta$ for all $\beta<\lambda$. 
  It follows that $E$ is a well-founded and extensional  relation on $\lambda$ and $N$ is the transitive collapse of $\langle\lambda,E\rangle$. 
 In addition, Proposition \ref{proposition:FixedSets} shows that $j^+(E)=E$ holds and therefore we can apply {\cite[Lemma 3.3]{MR3902806}} to conclude that $j\restriction\lambda$ is an elementary embedding of $\langle\lambda,E\rangle$ into itself. 
 It follows that $$\map{i=b\circ j\circ b^{{-}1}}{N}{N}$$ is an elementary embedding with $i(\lambda)=\lambda$, $i(\bar{U})=\bar{U}$ and $i\restriction\lambda=j\restriction\lambda$. 
 \[
 \begin{tikzcd}
{(N, \in)} \arrow[rr, "i"] \arrow[dd, "b^{-1}"] &  & {(N, \in)}                  \\
&  &                        \\
{\langle \lambda, E\rangle} \arrow[rr, "j"]     &  & {\langle \lambda, E\rangle} \arrow[uu, "b"']
\end{tikzcd}
 \]
Next, notice that our setup ensures that $$U ~ = ~ \Set{A\in\POT{\lambda}\cap W^j_\omega}{\exists m<\omega ~ \forall n \in[m,\omega) ~ \kappa_n\in A} ~ \in ~ W^j_\lambda$$ and $$\bar{U} ~ = ~ N\cap U ~ = ~ \Set{A\in\POT{\lambda}^N}{\exists m<\omega ~ \forall n \in [m,\omega) ~ \kappa_n\in A} ~ \in ~ N.$$
Elementarity then ensures that $\bar{U}$  is a normal ultrafilter on $\lambda$ in $N$. Let $\PPP_{\bar{U}}^N$ denote  Prikry forcing with $\bar{U}$ in $N$. The fact that $i(\bar{U})=\bar{U}$ holds then ensures that $i(\PPP_{\bar{U}}^N)=\PPP_{\bar{U}}^N$ holds. 
 Let $G$ denote the filter on $\PPP_{\bar{U}}^N$ induced by $\vec{\kappa}$, {i.e.,} the filter $G$ consists of all conditions $\langle s,A\rangle$ in $\PPP_{\bar{U}}^N$ with the property that $s(m)=\kappa_m$ for all $m<\length{s}$ and $\kappa_m\in A$ for all $\length{s}\leq m<\omega$. 
 The above equalities now allow us to use the \emph{Mathias criterion} for Prikry forcing (see \cite[Theorem 18.7]{Kan:THI}) to conclude that $G$ is  $\PPP_{\bar{U}}^N$-generic over $N$ with $\vec{\kappa}\in N[G]$. 
 Next, let $H$ denote the filter on $\PPP_{\bar{U}}^N$ induced by the sequence $\seq{\kappa_{m+1}}{m<\omega}$, {i.e.,} the filter $H$ consists of all conditions $\langle s,A\rangle$ in $\PPP_{\bar{U}}^N$ with the property that $s(m)=\kappa_{m+1}$ for all $m<\length{s}$ and $\kappa_{m+1}\in A$ for all $\length{s}\leq m<\omega$. 
 It then follows that $H$ is also $\PPP_{\bar{U}}^N$-generic over $N$ with $N[G]=N[H]$.  
Since $i[G]\subseteq H$ holds, standard arguments allow us to find an elementary embedding 
\begin{align*}
i_*: N[G] &\to N[G]\\
\tau^G &\mapsto i(\tau)^H
\end{align*}
extending $i$ to $N[G]$.

By our initial assumption that $\kappa_0$ is not exacting after Prikry forcing with $U_0$, elementarity now implies that $\lambda$ is not an exacting cardinal in $N[G]$. 
An application of Theorem \ref{lemma:exact_char} (see also Remark \ref{remark2})  now shows that, in $N[G]$, there is a non-empty subset $A$ of $V_{\lambda+1}$ that is definable by a formula with parameter $\lambda$ and has the property that for all $x,y\in A$, there is no non-trivial elementary embedding of $\langle V_\lambda,\in,x\rangle$ into $\langle V_\lambda,\in,y\rangle$. Pick $x\in A$ and set $y=i_*(x)$. 
Since $i_*(\lambda)=\lambda$ and $A$ is definable from $\lambda$, we have $i_*(A)=A$ and hence $y\in A$. 

Define $T$ to be the set of all non-trivial partial elementary embeddings $p$ of $\langle V_\lambda,\in,x\rangle$ into $\langle V_\lambda,\in,y\rangle$ with $\dom{p}=V_{\kappa_m}$ and $\ran{p}\subseteq V_{\kappa_{m+1}}$ for some $0<m<\omega$.
The fact that the sequence $\vec{\kappa}$ is an element of $N[G]$ then implies that the set $T$ is also an element of $N[G]$. 
Moreover, if we order the elements of $T$ by inclusion, then we obtain a tree of height at most $\omega$. 
It is now easy to see that for all $0<m<\omega$, the map $i_*\restriction V_{\kappa_m}$ is an element of the $(m-1)$-th level of $T$. This shows that the tree $T$ has height $\omega$ and it contains a cofinal branch in $V$. 
Since a sufficiently strong fragment of $\ZFC$ holds in $N[G]$, we now know that there is a cofinal branch $B$ through $T$ in $N[G]$. But, this implies that $\bigcup B$ is a non-trivial elementary embedding of $\langle V_\lambda,\in,x\rangle$ into $\langle V_\lambda,\in,y\rangle$ in $N[G]$, which is a contradiction.  This proves the theorem.
 \end{proof}

 \begin{corollary}
  The existence of an $\mathrm{I3}_{\rm{wf}(0)}$-embedding implies the existence of a transitive model of $\ZFC$ together with Vop\v{e}nka's Principle and the existence of an exacting cardinal. 
 \end{corollary}
 
\begin{proof}
  Let $\map{j}{V_\lambda}{V_\lambda}$ be an $\mathrm{I3}_{\rm{wf}(0)}$-embedding with critical sequence $\seq{\kappa_m}{m<\omega}$ and let $U=\Set{A\subseteq\kappa_0}{\kappa_0\in j(A)}$. 
  Then $\PPP_U\in V_\lambda$. Let $G$ be $\PPP_U$-generic over $V_\lambda$.  Since $G$ is also $\PPP_U$-generic over $V$, Theorem \ref{theorem:PrikryExacting} shows that $\kappa_0$ is an exacting cardinal in $V[G]$. Moreover, the fact that $V_\lambda[G]=V[G]_\lambda$ ensures that $\kappa_0$ is also an exacting cardinal in $V_\lambda[G]$. Using the weak homogeneity of $\PPP_U$, we can now conclude that, in $V_\lambda$, every condition in $\PPP_U$ forces $\kappa_0$ to be an exacting cardinal. 
  
  Now work in $V$ again and let $X$ be a countable elementary submodel of $V_\lambda$ with $U\in X$. Let $\map{\pi}{X}{M}$ denote the corresponding transitive collapse and let $H$ be $\pi(\PPP_U)$-generic over $M$. The above arguments then show that $\pi(\kappa_0)$ is an exacting cardinal in $M[H]$. Moreover, since Vop\v{e}nka's Principle holds in $V_\lambda$, it also holds in any set-generic forcing extension of $V_\lambda$ (\cite[Theorem 14]{BT:VP}), and so it  holds in $M[H]$. 
\end{proof}

\begin{corollary}
\label{coro5.6}
If $\ZFC$ is consistent with the existence of an $\mathrm{I2}$-embed\-ding, then $\ZFC$ is consistent together with the $\HOD$ Hypothesis and the existence of an extendible cardinal above an exacting cardinal. 
\end{corollary}

\begin{proof}
 By {\cite[Theorem 1.7]{MR3320593}} and {\cite[Section 3]{MR3902806}}, the consistency of $\ZFC$ with the existence of an $\mathrm{I2}$-embedding implies the consistency of $\ZFC$ with an $\mathrm{I3}_1$-embedding and the assumption that for every inaccessible cardinal $\kappa$, there exists a well-ordering of $H_{\kappa^+}$ that is definable in $H_{\kappa^+}$ by a formula without parameters. 
 Work in a model of this theory and fix an $\mathrm{I3}_1$-embedding $\map{j}{V_\lambda}{V_\lambda}$ with critical sequence $\seq{\kappa_m}{m<\omega}$. It then follows that $V_\lambda$ is a model of $\ZFC$ in which both $V=\HOD$ and Vop\v{e}nka's Principle hold. In particular, the $\HOD$ Hypothesis holds in $V_\lambda$. Set $U=\Set{A\subseteq\kappa_0}{\kappa_0\in j(A)}$ and let $G$ be $\PPP_U$-generic over $V$. Then $V_\lambda[G]$ is  a model of both the $\HOD$ Hypothesis and Vop\v{e}nka's Principle, because both principles are preserved by set-sized forcings (see {\cite[Corollary 8]{MR3821636}}). Finally, Theorem \ref{theorem:PrikryExacting} shows that, in $V_\lambda[G]$, there is an extendible cardinal above an exacting cardinal.  
\end{proof}

\subsection{Exacting cardinals and I3 embeddings}
In the remainder of this section, we consider lower bounds for the consistency strength of exacting cardinals.

\begin{proposition}
\label{exactinglowerbound}
 If $\lambda$ is an exacting cardinal and $\gamma<(\lambda^+)^{\HOD_{V_\lambda}}$, then there is an $\mathrm{I3}$-embedding $\map{j}{V_\lambda}{V_\lambda}$ with critical point $\kappa$ that satisfies $\gamma\in W^j_\omega\cap j_{0,\omega}[\kappa^+]$. 
\end{proposition}
 
\begin{proof}
 Assume, towards a contradiction, that the above implication fails for some exacting cardinal $\lambda$ and let $\gamma<(\lambda^+)^{\HOD_{V_\lambda}}$ be the minimal counterexample. Then there is $z\in V_\lambda$ with $\gamma<(\lambda^+)^{\HOD_{\{z\}}}$. Let $E$ be the minimal element in the canonical well-ordering of $\HOD_{\{z\}}$ that is a well-ordering of $\lambda$ of order-type $\gamma$. 
 Then both $\gamma$ and $E$ can be defined by formulas with parameters $\lambda$ and $z$. Pick $\zeta>\lambda$ such that $V_\zeta$ is sufficiently elementary in $V$. By our assumption, there is an elementary submodel $X$ of $V_\zeta$ with $V_\lambda\cup\{\lambda\}\subseteq X$  and an elementary embedding $\map{i}{X}{V_\zeta}$ with $i\restriction\lambda\neq\id_\lambda$, $i(\lambda)=\lambda$   and $i(z)=z$. It then follows that both $\gamma$ and $E$ are elements of $X$ with $i(\gamma)=\gamma$ and $j(E)=E$. 
 
  Now, define $\map{j=i\restriction V_\lambda}{V_\lambda}{V_\lambda}$. Then $j$ is an $\mathrm{I3}$-embedding and elementarity ensures that $$j^+\restriction(V_{\lambda+1}\cap X) ~ = ~ i\restriction(V_{\lambda+1}\cap X).$$ In particular, we have $j^+(E)=E$. Set $\kappa=\crit{j}$ and $E_0=E\cap V_\kappa$.  Then Proposition \ref{proposition:FixedSets} shows that $E=j_{0,\omega}(E_0)\in W^j_\omega$. 
  The elementarity of $j_{0,\omega}$ then implies that $E_0$ is a well-ordering of $\kappa$. Let $\gamma_0$ denote the order-type of this well-order.  
  By elementarity, in $M^j_\omega$ the set $E$ is a well-ordering of $\lambda$ of order-type $j_{0,\omega}(\gamma_0)$.   Since $\lambda$ and $E$ are elements of $W^j_\omega$ , we can now conclude that $j_{0,\omega}(\gamma_0)\in W^j_\omega$ and $j_{0,\omega}(\gamma_0)$ is also the order-type of $E$ in $V$. This shows that $$\gamma ~ = ~ j_{0,\omega}(\gamma_0) ~ \in ~  W^j_\omega\cap j_{0,\omega}[\kappa^+]$$ contradicting our initial assumption. 
\end{proof}

We will next show that the conclusion of Proposition \ref{exactinglowerbound} implies that there are many $\mathrm{I3}$-embeddings below $\lambda$. To this purpose, we shall first prove the following lemma, which we will apply to derive fragments of the Principle of Dependent Choice in models of the form $L(V_\lambda)$:

\begin{lemma}\label{lemma:FragmentsDC}
    Let $\lambda$ be a strong limit cardinal, let $M$ be an inner model of ZF with $V_\lambda\subseteq M$ and let $T\in M$ be a tree of height $\omega$ whose underlying set is a subset of $V_\lambda$. If $T$ has an infinite  branch in $V$ and $\lambda$ is regular  in $M$, then  $T$ has an infinite branch in $M$. 
\end{lemma}

\begin{proof}
 We start by proving two claims. 

 \begin{claim*}
   If $\alpha<\lambda$ and $\map{f}{D}{\lambda}$ is a function in $M$ with $D\subseteq V_\alpha$, then $\ran{f}$ is bounded in $\lambda$. 
 \end{claim*}

 \begin{proof}[Proof of the Claim]
  Since $\lambda$ is a strong limit cardinal in $V$, it follows that there is a wellordering of $D$ of order-type less than $\lambda$ in $V_\lambda$. The fact that $\lambda$ is regular now yields the statement of the claim.   
 \end{proof}

 \begin{claim*}
  In $M$, there is a non-empty pruned (i.e., with no maximal elements) subtree  of $T$. 
 \end{claim*}

 \begin{proof}[Proof of the Claim]
  Given a tree $S$, we let $S^\prime$ denote the subtree of $S$ consisting of all non-maximal elements of $S$. Now, let $\seq{T_\alpha}{\alpha\in\Ord}$ denote the unique sequence with $T_0=T$, $T_{\alpha+1}=T_\alpha^\prime$ for all $\alpha\in \Ord$ and $T_\lambda=\bigcap_{\alpha<\lambda}T_\alpha$ for every limit ordinal $\lambda$. Then there exists an ordinal $\beta$ with $T_\beta=T_{\beta+1}$. 
  Since $T$ has an infinite branch in $V$, the elements of this branch are contained in $T_\alpha$ for every ordinal $\alpha$. Thus, $T_\beta$ is a non-empty pruned subtree of $T$ that is an element of $M$. 
 \end{proof}

  Now work in $M$ and fix a non-empty pruned subtree $S$ of $T$. Let $S(n)$ be the $n$-th level of $S$. By our first claim, for all $\alpha<\lambda$ and all $n<\omega$, there exists $\alpha<\beta<\lambda$ with the property that for all $s\in S(n)\cap V_\alpha$, there exists $t\in S(n+1)\cap V_\beta$ with $s<_S t$. We may now  recursively define a strictly increasing sequence $\seq{\alpha_n}{n<\omega}$ of ordinals below $\lambda$ such that $S(0)\cap V_{\alpha_0}\neq\emptyset$ and for every $n<\omega$ and every $s\in S(n)\cap V_{\alpha_n}$, there exists $t\in S(n+1)\cap V_{\alpha_{n+1}}$ with $s<_S t$. Define $U$ to be the subtree of $S$ with underlying set $\bigcup\Set{S(n)\cap V_{\alpha_n}}{n<\omega}$. 

  Set $\alpha=\sup_{n<\omega}\alpha_n$. Our assumption on $\lambda$ implies  that $\alpha<\lambda$. Since  $U$ is non-empty and pruned, this tree contains a cofinal branch $b$ in $V$. But note that  $b\in V_{\alpha+1}\subseteq V_\lambda\subseteq M$.  
\end{proof}

We shall now derive the existence of many $\mathrm{I3}$-embeddings from the  conclusion of Proposition \ref{exactinglowerbound}.

\begin{proposition}
\label{prop5.9}
 Let $\lambda$ be a cardinal with the property that for every $\gamma<(\lambda^+)^{L(V_\lambda)}$, there exist an $\mathrm{I3}$-embedding $\map{j}{V_\lambda}{V_\lambda}$ with critical point $\kappa$ and $\delta_0<(\kappa^+)^{L(V_\kappa)}$  such that  $\gamma<j_{0,\omega}(\delta_0)\in W^j_\omega$. 
 Then $\lambda$ is regular in $L(V_\lambda)$, and for every closed unbounded subset $C$ of $\lambda$ in $L(V_\lambda)$ there is an $\mathrm{I3}$-embedding $\map{j}{V_{\lambda^\prime}}{V_{\lambda^\prime}}$ with $\lambda^\prime<\lambda$ whose critical sequence consists of elements of $C$.  
\end{proposition}
 
 \begin{proof}
  Fix a closed unbounded subset $C$ of $\lambda$ in $L(V_\lambda)$ of order-type $\cof{\lambda}^{L(V_\lambda)}$. Pick $\lambda<\gamma<(\lambda^+)^{L(V_\lambda)}$ with  $C\in L_\gamma(V_\lambda)$. 
  By our assumptions, we can find an $\mathrm{I3}$-embedding $\map{j}{V_\lambda}{V_\lambda}$ with critical sequence $\seq{\kappa_m}{m<\omega}$ and an ordinal $\delta_0<(\kappa_0^+)^{L(V_{\kappa_0})}$ with $\gamma<j_{0,\omega}(\delta_0)\in W^j_\omega$. 
Let $\vec{D}$ be an enumeration in length $\kappa_0$ of all closed unbounded subsets of $\kappa_0$ in $L_{\delta_0}(V_{\kappa_0})$,  and let $D_0$ be the diagonal intersection of $\vec{D}$. Set $\delta=j_{0,\omega}(\delta_0)$ and $D=j_{0,\omega}(D_0)$. We then know that $L_\delta(V_\lambda)\in W^j_\omega$. In particular, this implies that $\lambda$ is regular in $L_\delta(V_\lambda)$.
Moreover, elementarity ensures that $D$ is equal to a diagonal intersection of all closed unbounded subsets of $\lambda$ in $L_\delta(V_\lambda)$ and hence there exists $m<\omega$ with $D\cap[\kappa_m,\lambda)\subseteq C$. Since $D_0$ has order-type $\kappa_0$, it  follows, by elementarity, that $D$, and therefore also $C$, have order-type $\lambda$. Hence, $\lambda$ is a regular cardinal in $L(V_\lambda)$. Finally, note that we have $\kappa_n\in D$ for all $n<\omega$ and this implies that $\kappa_n\in C$ for all $m\leq n<\omega$. This shows that $\map{j_{m,m+1}}{V_\lambda}{V_\lambda}$ is an $\mathrm{I3}$-embedding whose critical sequence $\seq{\kappa_{n}}{m\leq n<\omega}$ consists of elements of $C$.

Now define $T$ to be the set of all partial elementary embeddings $\map{i}{V_\lambda}{V_\lambda}$ with the property that there exists a natural number $0<\ell<\omega$ and a strictly increasing sequence $\seq{\lambda_k}{k\leq\ell}$ of elements of $C$ with $\dom{i}=V_{\lambda_{\ell-1}}$, $\ran{i}\subseteq V_{\lambda_\ell}$, $i\restriction\lambda_0=\id_{\lambda_0}$,  $i(\lambda_k)=\lambda_{k+1}$ and $V_{\lambda_k}\prec V_\lambda$ for all $k<\ell$. 
  Then $T$ is an element of $L(V_\lambda)$ and, if we order the elements of $T$ by inclusion, then we turn $T$ into a tree of height at most  $\omega$. Given $0<\ell<\omega$, it is now easy to see that the sequence $\seq{\kappa_{m+k}}{k\leq\ell}$ witnesses that $j_{m,m+1}\restriction V_{\kappa_{\ell-1}}$ is an element of the $(\ell-1)$-th level of $T$. 
  This shows that $T$ has an infinite branch in $V$, and since, as we showed, $\lambda$ is regular in $L(V_\lambda)$, we may apply Lemma \ref{lemma:FragmentsDC} to conclude that $T$ has an infinite branch $B$ in $L(V_\lambda)$. 
   Set $i=\bigcup B$. Then the definition of $T$ yields an ordinal $\lambda^\prime\leq\lambda$ of countable cofinality  with the property that $i$ is a non-trivial elementary embedding from $V_{\lambda^\prime}$ into itself whose critical sequence consists of elements of $C$. Since $i$ is an element of $L(V_\lambda)$ and $\lambda$ is regular in $L(V_\lambda)$, we have that $\lambda^\prime<\lambda$. 
\end{proof}

Theorem \ref{theorem:PrikryExacting} and Proposition \ref{exactinglowerbound} yield now the following chain of implications:

\begin{theorem}
The consistency of each of the following theories implies the consistency of the next one, modulo {\rm ZFC}:
\begin{enumerate}
\item There exists an ${\mathrm I2}$-embedding.
\item There exists an ${\mathrm I3}_{\rm{wf}(0)}$-embedding.
\item There exists an exacting cardinal.
\item There is a cardinal $\lambda$ which is regular in $L(V_\lambda)$, and such that in $L(V_\lambda)$ the set of cardinals that are the critical point of an ${\mathrm I3}$-embedding is stationary. 
\item There exists an ${\mathrm I3}$-embedding.
\end{enumerate}
\end{theorem}

\begin{proof}
That $(1)$ is strictly stronger than  $(2)$ follows from \cite[Section 3]{MR3902806}. 
Theorem \ref{theorem:PrikryExacting} shows that the consistency of $(2)$ implies that of $(3)$. That the consistency of $(3)$ implies that of $(4)$ is a consequence of Propositions \ref{exactinglowerbound} and \ref{prop5.9}, because the conclusion of \ref{exactinglowerbound} yields the assumption of \ref{prop5.9}, as $(\lambda^+)^{L(V_\lambda)} \leq (\lambda^+)^{{\mathrm HOD}_{V_\lambda}}$. Finally, $(4)$ is trivially strictly stronger than $(5)$, consistency-wise.
\end{proof}

\begin{figure}[h]
\begin{center}
\begin{tikzcd}
\text{I2} \arrow[r] & \text{I3$_1$} \arrow[r] & \text{I3$_{\rm{wf}(n+1)}$} \arrow[r] & \text{I3$_{\rm{wf}(0)}$} \arrow[r] & \text{Exacting} \arrow[r] & \text{I3}
\end{tikzcd}
\end{center}
    \caption{Large cardinals between I2 and I3, ordered by consistency strength. None of the arrows reverse.}
    \label{fig:enter-label}
\end{figure}

In particular, we obtain the following result which locates exacting cardinals within the hierarchy of traditional large cardinals:
\begin{corollary}
The consistency strength of an exacting cardinal is strictly between the existence of an ${\rm I2}$-embedding and an ${\mathrm I3}$-embedding.    
\end{corollary}
 

\section{Structural reflection}\label{SectSR}
In this last section we will give characterizations of exacting and ultraexacting cardinals in terms of Structural Reflection, thus showing that these cardinals fit nicely in the general framework of large cardinals as principles of Structural Reflection as presented in \cite{Ba:SR}. While such characterizations were already given in \cite{ABL}, the ones presented here are arguably more natural. As explained in the introduction, they allow recasting exactingness as a two-cardinal variant of unfoldability.

\subsection{Ultraexacting structural reflection} \label{SectUXSR}
We shall prove that the existence of an ultraexacting cardinal is equivalent to a simpler form of the  principle of \emph{Ultraexact Structural Reflection} from \cite{ABL}.  

First, recall that for  a  limit ordinal $\lambda$ and a function $f:V_\lambda\to V_\lambda$, a \emph{square root of $f$} is a function $r: V_\lambda \to V_\lambda$ with $r^+(r)=f$, where $r^+:V_{\lambda +1}\to V_{\lambda+1}$ is defined by $r^+(x)=\bigcup \{ r(x\cap V_\alpha):\alpha <\lambda\}$.

Given  a first-order language $\mathcal{L}$ containing a distinguished unary predicate symbol $\dot{P}$, we say that an $\mathcal{L}$-structure $A$ has \emph{type} $\langle \mu ,\lambda\rangle$ if the universe of $A$ has rank $\lambda$ and $\dot{P}^A$ has rank $\mu$.

 \begin{definition}
       Given  a first-order language $\mathcal{L}$ containing a unary predicate symbol $\dot{P}$, and given a class $\Ce$ of $\mathcal{L}$-structures, 
       the \emph{Ultraexacting Structural Reflection principle for  $\Ce$ at a cardinal $\lambda$} $(\UXSR_\Ce (\lambda))$  asserts that there is a function $f:V_\lambda\to V_\lambda$ and a cardinal $\mu <\lambda$ with the property that  for every structure $B$ in $\Ce$ of type  $\langle \mu, \lambda \rangle$, there exists a structure $A$ in $\Ce$ of type $\langle \nu ,\lambda \rangle$, for some $\nu <\mu$,  and a square root $r$ of $f$ such that the restriction of $r$ to the universe of $A$ is  an elementary embedding of $A$ into $B$.  
 \end{definition}

 The naturalness of the $\UXSR$ principle is illustrated in the following two propositions.

 \begin{prop}
 \label{naturalUXSR}
 The following are equivalent for a cardinal $\lambda$:
 \begin{enumerate}
     \item $\UXSR_\Ce (\lambda)$ holds for the class $\Ce$ of $\mathcal{L}$-structures of the form $\langle V_\xi ,\in , \alpha\rangle$, where $\alpha <\xi$. 
     \item There exists an elementary embedding $j:V_\lambda \to V_\lambda$.
 \end{enumerate}
 \end{prop}

 \begin{proof}
 %
 Assume (1), and let $f:V_\lambda \to V_\lambda$ and $\mu <\lambda$ witness $\UXSR_\Ce(\lambda)$. Then $\langle V_\lambda , \in , \mu\rangle \in \Ce$ is of type $\langle \mu ,\lambda \rangle$. So there is some $\langle V_\lambda ,\in, \nu\rangle\in \Ce$ with $\nu <\mu$ and a square root $r$ of $f$ such that the restriction $r\restriction V_\lambda: V_\lambda \to V_\lambda$ is an elementary embedding sending $\nu$ to $\mu$, which yields $(2)$.

 Now assume $j:V_\lambda \to V_\lambda$ is an elementary embedding, and let us show $(1)$. Set $f=j^2:V_\lambda \to V_\lambda$ and let $\mu$ be any element of the critical sequence of $j$ greater than the critical point of $j$. We claim that $f$ and $\mu$ witness $\UXSR_{\Ce}(\lambda)$. Note that there is only one element of $\Ce$ of type $\langle \mu,\lambda\rangle$, namely  $\langle V_\lambda, \in , \mu \rangle$. Letting $\nu =j^{-1}(\mu)$, we have that $\nu <\mu$, and  $j$, which is a square root of $f$, is an elementary embedding from $\langle V_\lambda, \in , \nu \rangle$ to $\langle V_\lambda, \in , \mu \rangle$, which yields $(1)$.
 \end{proof}

\begin{prop}
  \label{propPi1}
 The following are equivalent for a cardinal $\lambda$:
 \begin{enumerate}
     \item $\UXSR_\Ce(\lambda)$ holds for all classes $\Ce$ of $\mathcal{L}$-structures that are  $\Delta_1$ definable (i.e., both $\Sigma_1$ and $\Pi_1$ definable) using $V_\lambda$ as a parameter.
     \item There exists an elementary embedding $j:V_{\lambda +1} \to V_{\lambda +1}$.
 \end{enumerate}    
 \end{prop}

 We defer the proof of the proposition above, since it will follow from more general arguments given in the proofs of the next two lemmata.

\medskip

Following \cite[Section 4]{ABL}, let now $\mathcal{L}^\ast$ denote the first-order language that extends the language of set theory by a  unary function symbol $\dot{f}$, 
  a binary relation symbol $\dot{E}$, and a unary predicate symbol $\dot{P}$.  
 Define $\mathcal{U}$ to be the class of $\mathcal{L}^\ast$-structures $A$ such that  there exists a  limit cardinal $\lambda$  such that the following hold: 
 \begin{itemize}
     \item  The reduct of $A$ to the language of set theory is equal to $\langle V_\lambda,\in\rangle$.  

      \item If $\zeta$ is the least cardinal in $C^{(2)}$ greater than $\lambda$, then there is an elementary submodel $X$ of $V_\zeta$ with $V_\lambda\cup\{\lambda,\dot{f}^A\}\subseteq X$ and a bijection $\tau : X\to {V_\lambda}$ with $\tau(\lambda)=\langle 0,0\rangle$, $\tau(x)=\langle 1,x\rangle$ for all $x\in V_\lambda$ and 
     $$x\in y ~ \Longleftrightarrow ~ \tau(x) ~ \dot{E}^A ~ \tau(y)$$
      for all $x,y\in X$.  
 \end{itemize}
 It is easy to check that  $\mathcal{U}$ is a $\Delta_{3}$ class (i.e., both $\Sigma_3$ and $\Pi_3$ definable, without parameters). 
 Then, similarly as in \cite[Lemma 4.7]{ABL}, we have the following:

  \begin{lemma}
 \label{sqrtimpliesultraexact}
   If  $\lambda \in C^{(1)}$ and $\UXSR_{\mathcal{U}}(\lambda)$ holds,  then $\lambda$ is   ultraexacting. 
 \end{lemma}

 \begin{proof}
  Let $f: V_\lambda \to V_\lambda$ and $\mu <\lambda$ witness that $\UXSR_{\mathcal{U}}(\lambda)$  holds. 
  Let $\zeta$ be the least cardinal in $C^{(2)}$ greater than $\lambda$,  let $Y$ be an elementary submodel of $V_\zeta$ of cardinality $\lambda$ with $V_\lambda\cup\{\lambda,f\}\subseteq Y$, and let $\pi :Y \to V_\lambda$ be a bijection with $\pi(\lambda)=\langle 0,0\rangle$,  $\pi(x)=\langle 1,x\rangle$ for all $x\in V_\lambda$. 
  Then there is an $\mathcal{L}^\ast$-structure $B$ extending $\langle V_\lambda,\in\rangle$ with  $\dot{f}^B=f$, $\dot{E}^B=\{\langle \pi(x),\pi(y)\rangle : x,y\in Y, ~ x\in y\}$ and $\dot{P}^B=\mu$. 
  It follows that $B$ is an element of $\mathcal{U}$ of type $\langle \mu , \lambda\rangle$. Hence, there is a structure $A$ in $\mathcal{U}$ of type $\langle \nu , \lambda \rangle$, with $\nu <\mu$, and a square root $r: V_\lambda\to V_\lambda$ of $f$ that is an elementary embedding of $A$ into $B$. Notice that $r$ is an I3-embedding with $r(\nu)=\mu$. Also, we have that $r(\langle m,x\rangle)=\langle m,r(x)\rangle$ holds for all $x\in V_\lambda$ and $m<\omega$. 
 Since $A\in \mathcal{U}$, let $X$ be an elementary submodel of $V_\zeta$  with $V_\lambda\cup\{\lambda,\dot{f}^A\}\subseteq X$,  and let $\tau : X \to V_\lambda$ be a bijection such that $\tau(\lambda)=\langle 0,0\rangle$, $\tau(x)=\langle 1,x\rangle$ for all $x\in V_\lambda$, and $$x\in y ~ \Longleftrightarrow ~ \tau(x) ~ \dot{E}^A ~ \tau(y)$$
 holds for all $x,y\in X$. Now define $$j ~ := ~ \pi^{{-}1}\circ r\circ \tau : X \to V_\zeta.$$

We claim that $j$ is an ultraexact embedding at $\lambda$. First note that $$j(\lambda)  =  (\pi^{{-}1}\circ r\circ \tau)(\lambda)  =  (\pi^{{-}1}\circ r)(\langle 0,0\rangle)  =  \pi^{{-}1}(\langle 0,0\rangle)  =  \lambda$$ and 
  $$
      j(x)  =  (\pi^{{-}1}\circ r\circ \tau)(x)  =  (\pi^{{-}1}\circ r)(\langle 1,x\rangle )  =  \pi^{{-}1}(\langle 1,r(x)\rangle)  =  r(x)
 $$ holds for all $x\in V_\lambda$.  Further, $j$ is an elementary embedding:  for every $a\in X$ and every formula $\varphi(x)$ in the language of set theory, 
  $$X\!\models \varphi (x) \quad \mbox{iff} \quad \langle V_\lambda , \dot{E}^A\rangle \!\models \varphi (\tau (a))\quad \mbox{iff} \quad \langle V_\lambda , \dot{E}^B\rangle \models \!\varphi (r(\tau(a)))\quad \mbox{iff}$$
  $$Y\!\models \varphi (\pi^{-1}(r(\tau(a))))\quad \mbox{iff} \quad V_\zeta \models \varphi(j(a)).$$

  \begin{claim}
  \label{claim4.6}
      $j\restriction V_\lambda=\dot{f}^A$. 
  \end{claim}

  \begin{proof}[Proof of the Claim]
    Assume, towards a contradiction, that the  claim fails and pick $m<\omega$ with   $j\restriction  V_{\lambda_m}\neq\dot{f}^A\restriction V_{\lambda_m}$, where $\langle \lambda_m:m<\omega\rangle$ is the critical sequence of $r$. Elementarity  then implies that $$r(r\restriction V_{\lambda_m}) ~ = ~ r(j\restriction  V_{\lambda_m}) ~ \neq ~ r(\dot{f}^A\restriction V_{\lambda_m}) ~ = ~ \dot{f}^B\restriction V_{\lambda_{m+1}} ~ = ~ f\restriction V_{\lambda_{m+1}}.$$
    But since $r$ is a square root of $f$, we also have that $$r(r\restriction V_{\lambda_m}) ~ = ~ r^+(r)\restriction V_{\lambda_{m+1}} ~ = ~ f\restriction V_{\lambda_{m+1}}$$ which yields  a contradiction.  
  \end{proof}

   This completes the proof of the lemma, because  $j\restriction V_\lambda = \dot{f}^A\in X$, thus  showing that $j$ is an ultraexact embedding at $\lambda$. 
 \end{proof}

 The converse holds for all ordinal definable classes of $\mathcal{L}$-structures, namely,
 
 \begin{lemma}
 \label{lemmaUESR}
  If $\lambda$ is an ultraexacting cardinal, then for every ordinal definable class $\Ce$ of $\mathcal{L}$-structures,  the principle  $\UXSR_\Ce(\lambda)$  holds.
 \end{lemma}
 
 \begin{proof}
 Let $\Ce$ be an  ordinal definable class of $\mathcal{L}$-structures, and let $\Ce_{\lambda +1}=\Ce \cap V_{\lambda +1}$· Thus, $\Ce_{\lambda +1}$ is an element of $V_{\lambda +2}$ that is ordinal definable with $\lambda$ as an additional parameter.

By Lemma \ref{lemma:ultra_exact_char}, let $j:(V_{\lambda +1}, \Ce_{\lambda +1})\to (V_{\lambda +1}, \Ce_{\lambda +1})$ be an elementary embedding with critical point, $\kappa$,  less than $\lambda$. Let $\mu =j(\kappa)$, and let $f:=j\restriction V_\lambda : V_\lambda \to V_\lambda$. 
We claim that $\UXSR_\Ce (\lambda)$ holds, witnessed by $f$ and $\mu$. 
 
So suppose $A \in \Ce_{\lambda +1}$ is a structure of type $\langle \mu, \lambda \rangle$. 
Then the elementarity of $j$ implies that $j(A)\in \Ce_{\lambda +1}$, and  the restriction map $j\restriction A:A\to j(A)$ is an elementary embedding. Notice that $j\restriction A$ is the restriction to $A$ of the function $j\restriction V_\lambda$, which is a square root of $j(f)$.
 
We may now pull back the previous statement by $j^{-1}$ and use  elementarity to conclude that there is a  structure $B$ in $\Ce_{\lambda +1}$ of type $\langle \kappa ,\lambda\rangle$  and there exists an elementary embedding $i:B\to A$ that is the restriction to $B$ of a function that is a square root of $f$.
\end{proof}

Lemmas \ref{sqrtimpliesultraexact} and \ref{lemmaUESR} now yield  the following characterization of ultraexact cardinals in terms of Ultraexacting  Structural Reflection.

 \begin{theorem}\label{TheoremUltraexactingSR}
  A cardinal $\lambda$ is ultraexacting if and only if the principle $\UXSR_\Ce(\lambda)$ holds for all ordinal definable classes $\Ce$ (equivalently, for the particular $\Delta_3$ class used in the proof of Lemma \ref{sqrtimpliesultraexact}) of $\mathcal{L}$-structures.  
 \end{theorem}

Similar arguments as in the proofs of the two lemmata above yield a proof of Proposition \ref{propPi1}. Namely,

  \begin{proof}[Proof of Proposition \ref{propPi1}]
Let $\Ce$ be the class of $\mathcal{L}^\ast$ structures $A$ such that  there exists a  limit ordinal $\lambda$  such that the following hold: 
 \begin{itemize}
     \item  The reduct of $A$ to the language of set theory is equal to $\langle V_\lambda,\in\rangle$.  

      \item There is a transitive set $X$ with $V_\lambda\cup\{\lambda,\dot{f}^A\}\subseteq X$ and a bijection $\tau : X\to {V_\lambda}$ with $\tau(\lambda)=\langle 0,0\rangle$, $\tau(x)=\langle 1,x\rangle$ for all $x\in V_\lambda$ and 
     $$x\in y ~ \Longleftrightarrow ~ \tau(x) ~ \dot{E}^A ~ \tau(y)$$
      for all $x,y\in X$.  
 \end{itemize}
Thus, $\Ce$ is $\Delta_1$ definable with $V_\lambda$ as a parameter, since $A\in \Ce$ if and only if $M\models ``A\in \Ce"$, for every transitive model $M$ of a sufficiently-big finite fragment of ZFC that contains $V_\lambda \cup \{ \lambda\}$.

Assume $(1)$ and argue as in the proof of Lemma \ref{sqrtimpliesultraexact}, working with the class $\Ce$, instead of the class $\mathcal{U}$.
Right before Claim \ref{claim4.6}, we have the following:
$$X\!\models \varphi (x) \quad \mbox{iff} \quad \langle V_\lambda , \dot{E}^A\rangle \!\models \varphi (\tau (a))\quad \mbox{iff} \quad \langle V_\lambda , \dot{E}^B\rangle \models \!\varphi (r(\tau(a)))\quad \mbox{iff}$$
  $$Y\!\models \varphi (\pi^{-1}(r(\tau(a))))\quad \mbox{iff} \quad Y \models \varphi(j(a)).$$
 Since, by Claim \ref{claim4.6}, $j\restriction V_{\lambda}\in X$, we also have that $j\restriction (X\cap V_{\lambda +1})\in X$, because if $x\in X\cap V_{\lambda +1}$, then $j(x)=\bigcup \{j(x\cap V_{\lambda_m}):m<\omega\}$. This shows that in $X$ there exists an elementary embedding from $V_{\lambda +1}$ to itself. Hence, by the elementarity of $j$,  such an elementary embedding exists in $V$.


 For the converse, let $j : V_{\lambda+1}\to V_{\lambda+1}$ be an elementary embedding. Let $f:=j^2\restriction V_\lambda :V_\lambda \to V_\lambda$, and let $\mu$ be any cardinal in the critical sequence of $j$ greater than the critical point. We claim that $f$ and $\mu$ witness $\UXSR_{\Ce} (\lambda)$ for any class $\Ce$ of structures that is $\Delta_1$ definable with $V_\lambda$ as a parameter.
   
   So let   $\Ce$ be such a class and fix $A\in \Ce$ of type $\langle \mu ,\lambda \rangle$. Then  
   $$V_{\lambda +1}\models ``A\in \Ce"$$ by downward absoluteness for transitive classes. By elementarity, $$V_{\lambda +1}\models ``j(A)\in \Ce"$$ and  the restriction map $j\restriction A:A\to j(A)$ is an elementary embedding. Note that $j\restriction A$ is the restriction to $A$ of the function $j\restriction V_\lambda$, which is a square root of $j(f)$.
 
By pulling back the previous statement via $j^{-1}$  we have, by elementarity, that in $V_{\lambda +1}$ there is a  structure $B$ in $\Ce$ of type $\langle j^{-1}(\mu) ,\lambda\rangle$  together with an elementary embedding $i:B\to A$ that is the restriction to $B$ of a function that is a square root of $f$. Since $V_{\lambda +1}$ is correct about $B$ belonging to $\Ce$, this completes the proof.
 \end{proof}

\subsection{Exacting structural reflection}
We will next show that the following simpler form of Structural Reflection characterizes exacting cardinals.

\begin{definition}
       Given  a first-order language $\mathcal{L}$ containing a unary predicate symbol $\dot{P}$, and given a class $\Ce$ of $\mathcal{L}$-structures, 
       the \emph{Exacting Structural Reflection principle for  $\Ce$ at a cardinal $\lambda$} $(\XSR_\Ce (\lambda))$  asserts that there exists a cardinal $\mu <\lambda$ with the property that  for some structure $B$ in $\Ce$ of type  $\langle \mu, \lambda \rangle$, if there is any, there exists a structure $A$ in $\Ce$ of type $\langle \nu ,\lambda \rangle$, for some $\nu <\mu$,  and   an elementary embedding of $A$ into $B$.   
 \end{definition}

Observe  that if $\Ce$ is the class of structures of the form $\langle V_\xi ,\in ,\alpha \rangle$, where $\alpha <\xi$, then $\XSR_\Ce (\lambda)$ is equivalent to the existence of an elementary embedding $j:V_\lambda \to V_\lambda$ (see Proposition \ref{naturalUXSR}), hence  also equivalent to $\UXSR_\Ce (\lambda)$. 

 \medskip

Let now $\mathcal{L}^\ast$ denote the first-order language that extends the language of set theory by a  
  a binary relation symbol $\dot{E}$ and a unary predicate symbol $\dot{P}$.  
 Define $\mathcal{E}$ to be the class of $\mathcal{L}^\ast$-structures $A$ such that  there exists a  limit cardinal $\lambda$  such that the following hold: 
 \begin{itemize}
     \item  The reduct of $A$ to the language of set theory is equal to $\langle V_\lambda,\in\rangle$.  

      \item If $\zeta$ is the least cardinal in $C^{(2)}$ greater than $\lambda$, then there is an elementary submodel $X$ of $V_\zeta$ with $V_\lambda\cup\{\lambda\}\subseteq X$ and a bijection $\tau : X\to {V_\lambda}$ with $\tau(\lambda)=\langle 0,0\rangle$, $\tau(x)=\langle 1,x\rangle$ for all $x\in V_\lambda$ and 
     $$x\in y ~ \Longleftrightarrow ~ \tau(x) ~ \dot{E}^A ~ \tau(y)$$
      for all $x,y\in X$.  
 \end{itemize}
 It is easily seen that  $\mathcal{E}$ is a $\Delta_{3}$ class. 

 \begin{lemma}
 \label{srimpliesexact}
   If $\lambda \in C^{(1)}$ and  $\XSR_{\mathcal{E}}(\lambda)$ holds,  then $\lambda$ is  exacting. 
 \end{lemma}

 \begin{proof}
   Let $\mu <\lambda$  witness $\XSR_{\mathcal{E}}(\lambda)$. Let $\zeta$ be the least cardinal in $C^{(2)}$ greater than $\lambda$,  let $Y$ be an elementary submodel of $V_\zeta$ of cardinality $\lambda$ with $V_\lambda\cup\{\lambda\}\subseteq Y$, and let $\pi :Y \to V_\lambda$ be a bijection with $\pi(\lambda)=\langle 0,0\rangle$,  $\pi(x)=\langle 1,x\rangle$ for all $x\in V_\lambda$. 
  Then there is an $\mathcal{L}^\ast$-structure $B$ extending $\langle V_\lambda,\in\rangle$ with   $\dot{E}^B=\{\langle \pi(x),\pi(y)\rangle : x,y\in Y, ~ x\in y\}$ and $\dot{P}^B=\mu$. 
  It follows that $B$ is an element of $\mathcal{E}$ of type $\langle \mu , \lambda\rangle$.
  By $\XSR_{\mathcal{E}}(\lambda)$  there is a structure $A$ in $\mathcal{E}$ of type $\langle \nu , \lambda \rangle$, with $\nu <\mu$, and an elementary embedding of $A$ into $B$. 
 Since $A\in \mathcal{E}$, let $X$ be an elementary submodel of $V_\zeta$  with $V_\lambda\cup\{\lambda\}\subseteq X$,  and let $\tau : X \to V_\lambda$ be a bijection such that $\tau(\lambda)=\langle 0,0\rangle$, $\tau(x)=\langle 1,x\rangle$ for all $x\in V_\lambda$, and $$x\in y ~ \Longleftrightarrow ~ \tau(x) ~ \dot{E}^A ~ \tau(y)$$
 holds for all $x,y\in X$. Now letting $$i ~ := ~ \pi^{{-}1}\circ j\circ \tau : X \to V_\zeta.$$
we can easily check that $i$ is an exact embedding at $\lambda$. 
 \end{proof}

 The converse holds for all ordinal definable classes of  structures in a language containing a unary predicate symbol. Namely,
 
 \begin{lemma}
 \label{lemmaESR}
  Let $\mathcal{L}$ be a first-order language  containing a unary predicate symbol. If $\lambda$ is an exacting cardinal, then for every $n\geq 2$ there is a cardinal $\mu <\lambda$ such that for every  $\Sigma_n$-definable, with $\lambda$ as a parameter, class $\Ce$ of $\mathcal{L}$-structures,   the principle  $\XSR_\Ce(\lambda)$  holds, witnessed by $\mu$.
 \end{lemma}
 
 \begin{proof}
 Suppose $\lambda$ is exacting and let  $j:X\to V_\zeta$ be an elementary embedding witnessing  it, with $\zeta$ being the least element of $C^{(n)}$ greater than $\lambda$. Let $\kappa$ be the critical point of $j$, and let  $\mu=j(\kappa)$.

Let $\Ce$ be a $\Sigma_n$-definable, with $\lambda$ as a parameter, class of $\mathcal{L}$-structures and suppose there exists $B\in \Ce$ of type $\langle  \mu ,\lambda\rangle$. Then this is true in $V_\zeta$, and by elementarity there must exist $A\in \Ce$ of type $\langle \kappa ,\lambda\rangle$ in $X$, and therefore also in $V$.  
Now note that  the restriction embedding $j\restriction A:A\to j(A)$ is  elementary, with $j(A)$ being in $\Ce$ and of type $\langle \mu , \lambda\rangle$, so it witnesses $\XSR_\Ce(\lambda)$. 
\end{proof}

Lemmas \ref{sqrtimpliesultraexact} and \ref{lemmaUESR} now yield  the following characterization of exacting cardinals in terms of Exacting  Structural Reflection.

\begin{theorem}\label{TheoremExactingSR}
A cardinal $\lambda$ is exacting if and only if the principle $\XSR_\Ce(\lambda)$ holds for all definable, with parameter $\lambda$, classes $\Ce$ (equivalently, for the particular $\Delta_3$ class used in the proof of Lemma \ref{srimpliesexact}) of $\mathcal{L}$-structures. \end{theorem}

We conclude with another characterization of exacting cardinals in terms of Structural Reflection, which may be seen as a two-cardinal version of the characterization of $C^{(n)}$-strongly unfoldable cardinals given in \cite{BL2} (see also \S\ref{SectIntro}, where this characterization is re-stated), and which bears some similarity with the J\'onsson-like characterization from \cite{ABL}. Namely,

\begin{theorem}\label{TheoremExactingUnfoldable}
Let $n\geq 2$. A cardinal $\lambda$ is exacting if and only if for some $\mu$, for every class of structures $\mathcal{C}$ of the same signature, which is $\Sigma_n$-definable from parameters in $V_\mu \cup \{\lambda\}$, and every $B\in \mathcal{C}$ of type $\langle\mu,\lambda\rangle$, there is $A\in \mathcal{C}$ of type $\langle\nu,\lambda\rangle$ with \(\nu < \mu\) and an elementary embedding $j: A\to B$.   
\end{theorem}
\proof
For \(n \geq 3\), the fact that the hypothesis of the theorem implies that $\lambda$ is exacting follows from Theorem \ref{TheoremExactingSR}. 
To prove the optimal result when \(n = 2\), suppose towards a contradiction that \(A\subseteq V_{\lambda+1}\) is the \(\OD\)-least counterexample to Definition \ref{def:exacting-zf}.
Fix some \(\mu < \lambda\) witnessing the hypothesis of the theorem, and let \(\mathcal C\) be class of structures \((V_{\lambda},\gamma,y)\) for \(\gamma < \lambda\) and \(y\in A\). Fix any structure 
\((V_{\lambda},\mu,y)\in \mathcal C\). By hypothesis, there is some 
\((V_{\lambda},\nu,x)\in \mathcal C\) and an elementary embedding
\(j : (V_{\lambda},\nu,x)\to (V_{\lambda},\mu,y).\) In particular, \(j\) is nontrivial and \(x,y\in A\); this contradicts that \(A\) is a counterexample to Definition \ref{def:exacting-zf}.

Conversely, fix $n$ and let let $\mu$ be the critical point of an exact embedding $j: X \to V_\zeta$, where $V_\zeta$ is sufficiently elementary in $V$. We claim that this $\mu$ witnesses the Structural Reflection property in the statement of the theorem. Suppose towards a contradiction that $\mathcal{C}$ is a counterexample. So by the elementarity of $X$ in $V$, we can find $B$ in $X$ so that 
\begin{itemize}
    \item $B \in \mathcal{C}$, both in V and in X,
    \item $B$ is of type $(\mu, \lambda)$,
    \item there is no $i: A \to B$, where $A$ is in $\mathcal{C}$ and $A$ is of type $(\nu, \lambda)$ for some $\nu<\mu$.
\end{itemize}

By elementarity and the fact that $j$ fixes all parameters from the definition of $\mathcal{C}$, we have $j(\mathcal{C}) = \mathcal{C}$, and $j(B)$ witnesses that $\mathcal{C}$ is a counterexample to the Structural Reflection property, so we have the following in $V_\zeta$ (and thus in $V$ by elementarity):
\begin{itemize}
    \item $j(B) \in j(\mathcal{C}) = \mathcal{C}$, 
    \item there is no $i: A \to j(B)$, where $A$ is in $\mathcal{C}$ and $A$ is of type $(\nu, \lambda)$ for some $\nu<j(\mu)$.
\end{itemize}
However, taking $A = B$, $i = j \upharpoonright B$, and $\nu = \mu$, we obtain a contradiction.
\endproof

Ultraexacting cardinals admit a similar characterization in which the embedding $j:A\to B$ is required to be a square root of a fixed embedding $f:V_\lambda \to V_\lambda$, as can be seen by arguing as in \S\ref{SectUXSR}.

%
%
%
%
%

\section{Open questions}

\begin{question}\label{QuestionUEHOD}
Does the theory $\mathrm{ZFC}$ + ``there is an extendible cardinal above an ultraexacting cardinal'' disprove the HOD Hypothesis?
\end{question}
By \cite{ABL}, the theory $\mathrm{ZFC}$ + ``there is an extendible cardinal \textit{below} an ultraexacting cardinal'' disproves the HOD Hypothesis, while Corollary \ref{coro5.6} shows that an extendible above an exacting cardinal does not. 

In view of Theorem \ref{TheoremEquiconsistencyGeneralized}, a negative answer to  Question \ref{QuestionUEHOD} might require the construction of canonical inner models for extendible cardinals. Therefore, Question \ref{QuestionUEHOD} could serve as a test question for inner model theory,
similar in spirit to the question of whether \(\text{OD}_\mathbb R\) determinacy is consistent with an extendible cardinal. On the other hand, it is conceivable that Question \ref{QuestionUEHOD} could be resolved by forcing the HOD Hypothesis over a model with an extendible above an ultraexacting cardinal. This raises a basic question, having nothing to do with large cardinals: given an ordinal \(\alpha\), is there a forcing extension \(V[G]\) that preserves \(V_\alpha\) and does not change OD subsets of \(V_\alpha\) but \(V[G]\) satisfies that every set is ordinal definable from parameters in \({V_{\alpha+\omega}}\)?
\subsection*{Acknowledgements}
The authors would like to thank Hugh Woodin for discussions.
The work of the first-listed author was partially supported by FWF grants ESP-3N and STA-139. 
The work of the second-listed author was supported by the Generalitat de Catalunya (Catalonian Government) under
grant 2021 SGR 00348, and by the Spanish Government under grant PID2023-147428NB-I00.
Goldberg's research was supported by the National Science Foundation under Grant No. DMS-2401789.
The  fourth-listed author gratefully acknowledges support from the Deutsche Forschungsgemeinschaft (Project number 522490605).
%
%
%


\addtocontents{toc}{\protect\setcounter{tocdepth}{1}}
\bibliographystyle{alpha} 
\bibliography{masterbiblio}
\end{document}